\documentclass[10pt, oneside]{article}
\usepackage{mathrsfs}
\usepackage{amsmath,amsfonts,amssymb,amsthm}
\allowdisplaybreaks
\usepackage{mathtools}
\usepackage{caption}
\usepackage{subfigure}
\usepackage{cite}
\usepackage{graphicx}
\usepackage{float}
\usepackage{paralist}
\usepackage{indentfirst}
\usepackage{authblk}
\usepackage{tikz}
\usepackage{tikz-cd}
\usetikzlibrary{shapes, arrows, positioning}
\usepackage[colorlinks,
linkcolor=red,
citecolor=blue
]{hyperref}
\usepackage{color} 
\usepackage{extarrows}
\usepackage{dsfont}
\usepackage[T1]{fontenc}

\newtheorem{theorem}{Theorem}[section]
\newtheorem{lemma}[theorem]{Lemma}

\newtheorem{prop}{Proposition}[section]
\newtheorem{remark}{Remark}[section]

\numberwithin{equation}{section}

\usepackage[textheight=220truemm, textwidth=160truemm]{geometry}   
\usepackage[title,titletoc]{appendix}

\usepackage{titlesec}

\begin{document}

\title{{\LARGE \textbf{Global regularity and sharp decay rates to the 1D hypo-viscous compressible Navier-Stokes equations}}}

\author{Chen Liang, Zhaonan Luo, and Zhaoyang Yin}
\date{}

\maketitle

\renewcommand*{\Affilfont}{\small\it}

\begin{abstract}
In this paper, we study the global regularity and sharp decay rates for the isentropic hypo-viscous compressible Navier-Stokes equations in 1D. Firstly, we prove the global stability for the small initial data near a stable equilibrium. Especially, we establish the global critical regularity in the Sobolev space $H^{\beta}$ with $\frac{1}{2}<\beta<1$. Furthermore, by bootstrap argument, Fourier splitting method and energy method, we then establish the optimal time decay rates under the extra low-frequency smallness assumption. We find the $L^2$ energy is self-closed, which motivates us to obtain the existence of global large solutions for initial data with high regularity. By a pure energy method, we also derive the optimal time decay rates when $\frac{1}{2}\le\beta<\frac{3}{4}$. We find a phenomenon that $\|(a,u)\|_{L^2}$ still decays even if the initial data does not possess $L^2$ smallness. Notably, the low-frequency smallness assumption is removed in the case with $\frac{1}{2}\le\beta<\frac{3}{4}$.

\renewcommand{\thefootnote}{}

\footnotemark
\end{abstract}

\tableofcontents

\section{Introduction}\label{section1}

\subsection{Problem formulation}
In this paper, we consider the following compressible Navier-Stokes (CNS) equations on the $\mathbb{R}$:
\begin{align}\label{eq0}
\left\{\begin{aligned} 
&\partial_{t}\rho+(\rho u)_{x}=0,\\ 
&\partial_{t}(\rho u)+\mu(-\Delta)^{\beta}u+(\rho u^{2}+P(\rho))_{x}=0,\\
&\rho|_{t=0}=\rho_0,~~u|_{t=0}=u_0,
\end{aligned}\right.
\end{align}
where $\rho$ and $u$ are density and velocity, the pressure $P=\rho^{\gamma}\ (\gamma>1)$ is given a smooth function, $\mu>0$ is the coefficient of viscosity.
  Fractional Laplacian operator $(-\Delta)^{\beta}u$ is defined by the Fourier transform
	$$
	\widehat{(-\Delta)^{\beta}u}=|\xi|^{2\beta}\widehat{u},\quad\quad \widehat{u}(\xi,t)=\int_{\mathbb{R}}e^{-ix\cdot\xi}u(x,t) dx.
	$$
Fractional diffusion can describe behaviors such as jump L\'evy processes\cite{MR3185174}, radiation dynamics\cite{MR1031939}, and porous media mechanics\cite{MR2737788,MR2954615}.
Taking $\mu=1$ and $\rho=a+1$, (\ref{eq0}) can be written as
\begin{align}\label{eq1}
\left\{\begin{aligned} 
&\partial_{t}a+u_x=-(au)_x,\\ 
&\partial_{t}u+(-\Delta)^{\beta}u+\gamma a_x=K(a)a_x+G,\\
&a|_{t=0}=a_0,~~u|_{t=0}=u_0,
\end{aligned}\right.
\end{align}
where $$
K(a)=\gamma\frac{a}{1+a}+\frac{P^{\prime}(1)- P^{\prime}(1+a)}{1+a}, \\
$$
and
$$
G=-uu_x+\frac{a}{1+a}(-\Delta)^{\beta}u.
$$
The CNS equations with $\beta=1$ is the classical barotropic CNS equations. There are many results about classical CNS equations in multi-dimensions. Matsumura and Nishida \cite{1979Matsumura,1980Matsumura} established the global existence of small solutions and derived the corresponding time decay rate under an additional $L^1$ condition on the initial data. Further global well-posedness results for strong solutions are available in \cite{MR1779621,MR2675485,MR2679372,MR2847531}.  For the full CNS system without heat conduction,
 Z. Xin \cite{xin1988blowup} proved the non-existence of global non-trivial smooth solutions with compactly supported initial density. For the weak solutions of CNS equations, P. L. Lions \cite{1996Lions,1998Lions} obtained global finite-energy weak solutions. The related result was
later extended by  Feireisl et al.\cite{2004Feireisl}.

Impressive progress has been made in the time decay rates of multi-dimensional isentropic CNS equations with $\beta=1$. Under the assumption that the positive density tends to a positive constant (taken as 1 for simplicity) at infinity.  Ponce \cite{PONCE1985399} established the $L^{p}$ time decay rate with small initial data in $H^{s}\cap W^{s,1}(s>2+\frac{d}{2})$. Y. Guo and Y. Wang \cite{Guo01012012} obtained the optimal time decay rates for the CNS equations and the Boltzmann equation by a method purely based on energy estimates. R. Danchin and J. Xu \cite{MR3609245} proved the  time decay rates in the critical Besov space. Later, J. Xu \cite{Xu2019} established sharp time decay rates under a small low frequency assumption in some Besov spaces with negative index, which improved the results from \cite{MR3609245}. The restriction on the smallness of low frequency was removed by Z. Xin and J. Xu \cite{MR4188989}. Spectral analysis is an important method to analysis the large-time behavior. One may refer to \cite{Li2011Large,DUAN2007220,1979Matsumura} for more details.  

We now review some results about 1D CNS equations with $\beta=1$. Early mathematical results for 1D CNS equations can be found in  \cite{MR283426,MR421305,DavidHoff,Kanel1979,Kazhikhov1979,KAZHIKHOV1977273}. A. Mellet and A. Vasseur \cite{MR2368905} proved there exists a unique global strong solution in Sobolev space $H^1$. J. Li \cite{18M1167905} showed the global well-posedness for CNS equations with constant heat conductivity. Moreover, the results remain valid even in the presence of vacuum. The global well-posedness for entropy-bounded solution was also considered by J. Li and Z. Xin \cite{MR4039142,MR4491875}. K. Chen, L. K. Ha, R. Hu, and Q.-H. Nguyen \cite{MR4659290} proved the global well-posedness for rough initial data. Regarding large-time behavior, the large-time behavior for 1D viscous polytropic ideal gas has been studied in \cite{MR971685,MR1370096,MR1671920}. J. Li and Z. Liang \cite{Li2016Some} proved the solutions of CNS equations tend to 0 in $L^p$ norm. However, they did not provide an explicit decay rate. Subsequently, Y. Chen, M. Li, Q. Yao and Z. Yao  \cite{MR4549954} obtained the optimal time decay rates in $H^2$ by virtue of spectral analysis and Fourier splitting method. For more results about time decay rates for 1D CNS equations, we refer the reader to \cite{MR1890882,MR2812966}.

When $0<\beta<1$, Y. Li, P. Qu, Z. Zeng and D. Zhang \cite{li2022} proved there exist infinitely many weak solutions for hypo-viscous CNS equations. S. Wang and S. Zhang \cite{MR4643428} studied the $L^2$ time decay rate of (\ref{eq0}) with the case $\frac{1}{2}<\beta\le1$ in $\mathbb{R}^3$, where the initial data belongs to $H^{4}\cap L^{1}$. Moreover, S. Wang and S. Zhang \cite{MR4897629} obtained the time decay rate for the 3D Navier-Stokes-Poisson system with fractional dissipation, which can be regarded as an extension from  \cite{MR4643428}. Very recently, C. Liang, Z. Luo and Z. Yin \cite{liang2026} established the global regularity and optimal time decay rates for 2D CNS equations with $\frac{1}{2}\le\beta<1$.

To our best knowledge, the global regularity and large time behavior for CNS equations  (\ref{eq0}) with $\frac{1}{2}\le\beta<1$ has not been
studied yet. In this paper, we mainly investigate the global stability and optimal time decay rates of global strong solutions for system (\ref{eq0}). Unlike the higher-dimensional cases, the one-dimensional setting exhibits its own inherent difficulties, which will be addressed in detail later.

\subsection{Notations and main results}

In this paper, we denote $\mathscr{F}^{-1}(f)$ the inverse Fourier transform of $f$ and $\Lambda f=(-\Delta)^{\frac{1}{2}}f$. We define the following energy and dissipation functionals for $(a,u)$:
\begin{align*}
E_{0}(t)&=\int_{\mathbb{R}}\left(2G(\rho)+\rho u^2\right)dx,\quad E_{s}(t)=\left\|(\sqrt{\gamma} a, u)\right\|_{\dot{H}^{s}}^{2},\\ 
\mathscr{E}_{0}(t)&=E_{0}+E_{s}+\langle u,2k\Lambda^{-2+2 \beta} a_x\rangle_{H^{s-2\beta+1}},\\
\mathscr{D}_{0}(t)&=k\gamma\|\Lambda^{\beta}a\|_{H^{s-2\beta+1}}^{2}+\|\Lambda^{\beta} u\|_{H^{s}}^{2},\\
\mathscr{E}_{1}(t)&= E_{1}+E_{1+s}+\langle\Lambda^{\beta}a_x,2k\Lambda^{\beta}u\rangle_{H^{s-2\beta+1}},\\
\mathscr{D}_{1}(t)&=k\gamma\|\Lambda^{1+\beta}a\|_{H^{s-2\beta+1}}^{2}+\|\Lambda^{1+\beta} u\|_{H^{s}}^{2},
\end{align*} 
where $G(\rho)$ is the following potential energy density:
\begin{align*}
G(\rho)\overset{\mathrm{def}}{\operatorname*{=}}\rho\int_{1}^{\rho}\frac{s^{\gamma}-1}{s^2}ds.
\end{align*}
and $k>0$ is a sufficiently small constant.

Our main results are as follows:
\begin{theorem}\label{1theo1}
Let $\frac{1}{2}<\beta<1$, $s\ge\beta$ or $\beta=\frac{1}{2}$, $s>\frac{1}{2}$. Let $(a,u)$ be a local strong solution of (\ref{eq1}) with the 
initial data $(a_{0},u_{0})\in H^{s}$. There exists a small constant $\delta $ such that if 
\begin{align*}
\|(a_{0},u_{0})\|_{H^{s}}\leq\delta,
\end{align*}
then the system (\ref{eq1}) admits a unique global strong solution $(a,u)\in C([0,\infty),H^{s})$. Moreover, we obtain that for all $t>0$, there holds
\begin{align}\label{1201}
\quad\frac{d}{dt}\mathscr{E}_{0}+\mathscr{D}_{0}\leq 0.
\end{align}
\end{theorem}
\begin{remark}
It should be noticed that there exists parabolic effects for $a$ and $u$. The phenomenon is unique to the case with $\frac{1}{2}\le\beta<1$.
\end{remark}
\begin{theorem}\label{1theo2}
Let $\frac{1}{2}<\beta<1$, $s\ge\beta$ or $\beta=\frac{1}{2}$, $s>\frac{1}{2}$. Let $(a,u)$ be a strong solution of (\ref{eq1}) with the initial data  $(a_{0}, u_{0})$ under the condition in Theorem \ref{1theo1}. Suppose $\left(a_{0}, u_{0}\right) \in \dot{B}_{2, \infty}^{-\frac{1}{2}}$ and there exists a small constant $\delta $ such that if 
  \begin{align*}
 \|(a_{0}, u_{0})\|_{H^{s}}+\|(a_{0}, u_{0})\|_{\dot{B}^{-\frac{1}{2}}_{2,\infty}}\le\delta,
  \end{align*}
  then there exists $C>0$ such that for every $t>0$, there holds
\begin{align*}
\|\Lambda^{s_1}(a,u)\|_{L^{2}}\le C(1+t)^{-\frac{1+2s_1}{4\beta}},
\end{align*}
where $0\le s_1\le s$. If there exists a small constant $\delta $ such that 
  \begin{align*}
 \|(a_{0}, u_{0})\|_{H^{s+\beta}}+\|(a_{0}, u_{0})\|_{\dot{B}^{-\frac{1}{2}}_{2,\infty}}\le\delta,
  \end{align*} 
 and additionally $0<\left|\int_{\mathbb{R}}\left(a_{0}, u_{0}\right) d x\right|$, then there exists $C_{\beta} \leq C $ such that
\begin{align*}
\left\|\Lambda^{s_1}(a, u)\right\|_{L^{2}} \geq \frac{C_{\beta}}{2}(1+t)^{-\frac{1+2s_1}{4\beta}},
\end{align*}
where $0\le s_1\le s$.
\end{theorem}
\begin{remark}
  Note $L^1\hookrightarrow\dot{B}_{2, \infty}^{-\frac{1}{2}}$, it follows that the above results still hold true when $\left(a_{0}, u_{0}\right)\in L^1$ and
  \begin{align*}
  \|(a_0,u_0)\|_{L^1}\le\delta.
  \end{align*}
\end{remark}
\begin{remark}
By Sobolev interpolation, $L^{p}(p\ge2)$ time decay rates for $(a,u)$ may be obtained.
\end{remark}
\begin{theorem}\label{1theo3}
Let $\frac{1}{2}<\beta<\frac{3}{4}$, $s\ge1+\beta$ or $\beta=\frac{1}{2}$, $s>\frac{3}{2}$. Let $(a,u)$ be a local strong solution of (\ref{eq1}) with the 
initial data $(a_{0},u_{0})\in H^{s}$. For any positive constant $M$, there exists a small constant $\delta=\delta(M) $ such that if 
\begin{align*}
\|(a_{0},u_{0})\|_{L^{2}}\leq M,\quad\|\partial_{x}(a_{0},u_{0})\|_{H^{s-1}}\leq\delta(M),
\end{align*}
then the system (\ref{eq1}) admits a unique global strong solution $(a,u)\in C([0,\infty),H^{s})$. Moreover, we obtain that for all $t>0$, there holds
\begin{align}\label{1202}
\quad\frac{d}{dt}\mathscr{E}_{1}+\mathscr{D}_{1}\leq 0.
\end{align}
\end{theorem}
\begin{theorem}\label{1theo4}
Let $\frac{1}{2}<\beta<\frac{3}{4}$, $s\ge1+\beta$ or $\beta=\frac{1}{2}$, $s>\frac{3}{2}$. Let $(a,u)$ be a strong solution of (\ref{eq1}) with the initial data  $(a_{0}, u_{0})$ under the condition in Theorem \ref{1theo3}. If additionally $\left(a_{0}, u_{0}\right) \in \dot{B}_{2, \infty}^{-\frac{1}{2}}$, then there exists  $C>0$  such that for every $t>0$ and $s_1\in [0,s]$,
\begin{align*}
\left\|\Lambda^{s_1}(a,u)\right\|_{L^{2}} \leq C(1+t)^{-\frac{1+2s_1}{4\beta}}.
\end{align*}
\end{theorem}
\begin{remark}
From Theorem \ref{1theo2}, we know the time decay rates established in Theorem \ref{1theo4} are optimal.
\end{remark}
\subsection{Motivations and main ideas}
\textbf{Case1: Global regularity and time decay rates for small solutions to the CNS equations}

We firstly focus on establishing the global regularity for (\ref{eq1}). When $\frac{1}{2}<\beta<1$, we can get the global regularity in critical Sobolev spaces $H^{\beta}$. If $\beta=\frac{1}{2}$, due to the Sobolev embedding $H^{\frac{1}{2}+}\hookrightarrow L^{\infty}$, we need to require regularity index $s>\frac{1}{2}$. Then we can use standard bootstrap argument to get the global regularity results. More details can be found in Section \ref{section3}. The global regularity for $\beta=\frac{1}{2}$ in critical space remains an open problem. We conjecture that the critical space for $\beta=\frac{1}{2}$ is $B_{2,1}^{\frac{1}{2}}$, which can be embedded into the Sobolev space $H^{\frac{1}{2}}$.

Next, we aim to study the long time behavior of CNS equations (\ref{eq1}). In contrast to the high-dimensional cases, we cannot directly use the Fourier splitting method to obtain the initial decay rate since the decay may be slower. For convenience of explanation, we consider the case where $s=\beta$ and $\frac{1}{2}<\beta<1$.  A direct application of Schonbek's strategy \cite{Schonbek1985,Schonbek1991} only yields the following estimate:
  \begin{align*}
  \frac{d}{dt}\mathscr{E}_{0}(t)+\frac{C_{2}}{1+t}\left(k\gamma\|a\|_{H^{\beta}}^{2}+\|u\|_{H^{\beta}}^{2}\right)\leq C(1+t)^{-\frac{1}{4\beta}-\frac{1}{2}}.
  \end{align*}  
  This implies that
    \begin{align*}
\mathscr{E}_{0}(t)\leq C(1+t)^{-\frac{1}{4\beta}+\frac{1}{2}}.
  \end{align*}  
 However, this estimate is not useful for deriving the time decay rate. Inspired by \cite{MR4549954}, we apply bootstrap method again to obtain the optimal time decay rate. We assume
  \begin{align}\label{1302}
  (1+t)^{\frac{1}{4\beta}}\|(a,u)\|_{L^{2}}+(1+t)^{\frac{1}{4\beta}+\frac{1}{2}}\|\Lambda^{\beta}(a,u)\|_{L^{2}}\le C_{0}\delta,
  \end{align}
  where $C_{0}$ is determined later. To validate the bootstrap argument here, we impose an additional smallness condition on the initial data, specifically requiring that
  \begin{align}\label{1303}
   \|(a_{0}, u_{0})\|_{\dot{B}^{-\frac{1}{2}}_{2,\infty}}\le\delta. 
   \end{align}
   By virtue of (\ref{1302}), we obtain the key time weighted integrability. Subsequently, we may prove the propagation of negative regularity for the solutions obtained in Theorem \ref{1theo1}. Through precise estimates, we can deduce that
   \begin{align*}
  (1+t)^{\frac{1}{4\beta}}\|(a,u)\|_{L^{2}}+(1+t)^{\frac{1}{4\beta}+\frac{1}{2}}\|\Lambda^{\beta}(a,u)\|_{L^{2}}\le \frac{C_{0}\delta}{2}.
  \end{align*}
  Hence, we derive the upper bound of time decay rates for CNS equations (\ref{eq1}). A natural question then arises: How to remove the smallness assumption for the negative index Besov space $\dot{B}_{2,\infty}^{-\frac{1}{2}}$ under low regularity. Finally,  we introduce a new weighted energy estimate instead of complex spectral analysis to prove the lower bound of the decay rate. The detailed proofs can be found in Section \ref{section3}. As can be seen from the proof, our methods for deriving the decay rate does not apply to the case with $\beta=1$.  This seems to be a challenging problem in obtaining the optimal decay rates for the classical CNS equations.
  
  \textbf{Case2: Global regularity and time decay rates for large solutions to the CNS equations}
  
For simplicity, we restrict our attention to the case with $\frac{1}{2}<\beta<1$. It should be noticed that 
\begin{align}\label{1304}
\frac{d}{dt}\int_{\mathbb{R}}\left(G(\rho)+\frac{1}{2}\rho u^2\right)dx+\|\Lambda^{\beta}u\|_{L^2}^{2}=0,
\end{align}
which means that the $L^2$ energy is self-closed. This key observation motivates us to establish the global regularity of large solutions for CNS equations (\ref{eq1}). When testing the (\ref{eq1})$^{1}$ with $\Lambda^{2} a$, one may derive the following estimate from the right-hand side of (\ref{eq1})$^{1}$:
\begin{align}\label{1305}
-\langle\Lambda(au)_x,\Lambda a\rangle\le\ &C\|\Lambda^{1+\beta}a \|_{L^{2}}\|\Lambda^{2-\beta}(a,u)\|_{L^{\frac{1}{\frac{3}{2}-2\beta}}}\|(a,u)\|_{L^{\frac{1}{2\beta-1}}}\\\nonumber
\le\ &C\|\Lambda^{1+\beta}a\|_{L^{2}}\|\Lambda^{1+\beta}(a,u)\|_{L^{2}}\|\Lambda^{\frac{3}{2}-2\beta}(a,u)\|_{L^{2}}.
 \end{align}
 From (\ref{1305}),  we see that $\beta<\frac{3}{4}$. If $\beta=\frac{1}{2}$, owing to $H^{\frac{3}{2}+}\hookrightarrow C^{0,1}$, we require regularity index $s>\frac{1}{2}$. Then we yield the global regularity with large data.
 Our next goal is to establish the time decay rates with the low-frequency smallness assumption removed. As observed in the previous analysis, it is difficult to obtain an initial $L^2$ time decay rate in the one-dimensional case. Motivated by \cite{Guo01012012}, to begin with, we derive a Lyapunov-type inequality by applying the Sobolev interpolation and (\ref{1202}):
 \begin{align*}
\frac{d}{dt}\mathscr{E}_{1}+C\mathscr{E}_{1}^{1+\beta}\le 0.
\end{align*}
Directly calculation yields 
\begin{align}\label{1306}
\mathscr{E}_{1}\le C(1+t)^{-\frac{1}{\beta}}.
\end{align}
 This decay rate acts as a bootstrap that facilitates the derivation of the initial $L^2$ time decay rate. By virtue of improved Fourier splitting method and the Littlewood-Paley decomposition theory, we may obtain the optimal time decay rates through a delicate iterative argument. In what follows, we use a flow chart to depict the general proof strategy of Theorem \ref{1theo4}.
\begin{align*}
	\begin{split}&\boxed{\frac{d}{dt}\mathscr{E}_{0}+\mathscr{D}_{0}\le 0}\xrightarrow{(\ref{1306})}\boxed{\mathscr{E}_{0}\le C(1+t)^{-\frac{3}{4\beta}+1}}\xrightarrow{\text{First Iteration}}\boxed{\mathscr{E}_{0}\le C(1+t)^{-\frac{1}{4\beta}}}\\
&\xrightarrow{(\ref{1202})}\boxed{\mathscr{E}_{1} \le C(1+t)^{- \frac{5}{4\beta}}}\longrightarrow\boxed{\|(a,u)\|_{\dot{B}_{2,\infty}^{-\frac{1}{2}}}^{2}\le C}\\ 
& \xrightarrow{\text{Second Iteration}}\boxed{\mathscr{E}_{0}\le C(1+t)^{-\frac{1}{2\beta}}} \longrightarrow\boxed{\mathscr{E}_{1}\le C(1+t)^{-\frac{3}{2\beta}}}\\
& \longrightarrow\boxed{(1+t)^{-1}\int_{0}^{t}(1+t^{\prime})^{\frac{3}{2\beta}+1}\mathscr{D}_{1}dt^{\prime}\le C}\rightarrow\boxed{E_{1+\beta}\le C(1+t)^{-\frac{3}{2\beta}-1}}\end{split}
\end{align*}
 Comparing with the Theorem \ref{1theo2}, the smallness assumption (\ref{1303}) has been removed. More technical details can
be found in Section \ref{section4}. This result implies that the $L^2$ energy of solutions to system (\ref{eq1}) still decays even if the initial $L^2$ energy is large. A challenging problem also remains here: How to obtain the optimal decay rates of large solutions for the case with $\frac{3}{4}\le\beta<1$.

\textbf{Organization of the paper}
The rest of the paper is organized as follows. In Section \ref{section2}, we introduce some lemmas which will be used
in the sequel. In Section \ref{section3}, we present the proofs for the global regularity and sharp decay rates with small data.
In Section \ref{section4}, we will prove the global regularity and sharp decay rates with large data.

\section{Preliminaries} \label{section2}
The Littlewood-Paley decomposition theory and Besov spaces are given as follows.
	\begin{lemma}\cite{Bahouri2011}\label{lemma1}
		Let $\mathscr{C}$ be the annulus $\{\xi\in\mathbb{R}:\frac 3 4\leq|\xi|\leq\frac 8 3\}$. There exists a radial function $\varphi$, valued in the interval $[0,1]$, belonging to $\mathscr{D}(\mathscr{C})$, and such that
		$$ \forall\xi\in\mathbb{R}\backslash\{0\},\ \sum_{j\in\mathbb{Z}}\varphi(2^{-j}\xi)=1,~~~ $$
		$$ |j-j'|\geq 2\Rightarrow\mathrm{Supp}\ \varphi(2^{-j}\cdot)\cap \mathrm{Supp}\ \varphi(2^{-j'}\cdot)=\emptyset, $$
		Moreover, we have
		$$ \forall\xi\in\mathbb{R}\backslash\{0\},\ \frac 1 2\leq\sum_{j\in\mathbb{Z}}\varphi^2(2^{-j}\xi)\leq 1.~~ $$
	\end{lemma}
			
		Let $u$ be a tempered distribution in $\mathcal{S}_{h}'(\mathbb{R})$. For all $j\in\mathbb{Z}$, define
		The homogeneous operators are defined by
		$$\dot{\Delta}_j u=\mathscr{F}^{-1}(\varphi(2^{-j}\cdot)\mathscr{F}u).$$
Then the Littlewood-Paley decomposition is given as follows:
$$
u=\sum_{j\in\mathbb{Z}}\dot{\Delta}_j u\quad  \mathrm{in}\quad  \mathcal{S}_{h}'(\mathbb{R}). 
$$		
		Let $s\in\mathbb{R}$ and $(p,r)\in[1,\infty]^2$. The homogeneous Besov space $\dot{B}^s_{p,r}$ is given as follows
	$$ \dot{B}^s_{p,r}=\left\{u\in \mathcal{S}_{h}'(\mathbb{R}):\|u\|_{\dot{B}^s_{p,r}}=\Big\|(2^{js}\|\dot{\Delta}_j u\|_{L^p})_j \Big\|_{l^r(\mathbb{Z})}<\infty\right\}.$$
 We introduce the Gagliardo-Nirenberg inequality of Sobolev type with
$d = 1$.
\begin{lemma}\label{lemma3}\cite{1959On}
For $d=1$, $p\in[2,+\infty)$ and $0\leq s,s_{1}\leq s_{2}$, there holds
\begin{align*}
\|\Lambda^{s}f\|_{L^{p}}\leq C\|\Lambda^{s_{1}}f\|_{L^{2}}^{1- \theta}\|\Lambda^{s_{2}}f\|_{L^{2}}^{\theta},
\end{align*}
where $0\leq\theta\leq 1$ and satisfies  
$$s+\frac{1}{2}-\frac{1}{p}=(1-\theta)s_{1}+\theta s_{2}.$$ 
Note that we also require that $0<\theta<1, 0\leq s_{1}\leq s$, when $ p=\infty$.
\end{lemma}

The following commutator estimates and product estimate are useful for energy estimate.
\begin{lemma}\label{lemma4}\cite{kato1}
Assume that $s>0$, $p, p_1, p_4 \in (1,\infty)$ and $\frac{1}{p}=\frac{1}{p_1}+\frac{1}{p_2}=\frac{1}{p_3}+\frac{1}{p_4}$, then we obtain
\begin{align*}
\|[\Lambda^{s},f]g\|_{L^{p}}\leq C\left(\|\Lambda^{s}f\|_{L^{p_{1}}}\|g\|_{L^{p_{2} }}+\|\nabla f\|_{L^{p_{3}}}\|\Lambda^{s-1}g\|_{L^{p_{4}}}\right).
\end{align*}
\end{lemma}
\begin{lemma}\label{lemma6}\cite{MR4884564}
Assume that $s\in(0,1)$, $p, p_1, p_2 \in (1,\infty)$ and $\frac{1}{p}=\frac{1}{p_1}+\frac{1}{p_2}$, then we obtain
\begin{align*}
\|[\Lambda^{s},f\partial_x ]g\|_{L^{p}}\leq C\|\Lambda^{s}g\|_{L^{p_{1}}}\|\nabla f\|_{L^{p_{2} }}.
\end{align*}
\end{lemma}
\begin{lemma}\label{lemma5}\cite{kato1}
Assume that $s>0$, $p, p_2, p_4 \in (1,\infty)$ and $\frac{1}{p}=\frac{1}{p_1}+\frac{1}{p_2}=\frac{1}{p_3}+\frac{1}{p_4}$, then we obtain
\begin{align*}
\|\Lambda^{s}\left(fg\right)\|_{L^{p}} \leq C\left(\|f \|_{L^{p_{1}}}\|\Lambda^{s}g\|_{L^{p_{2}}} + \|g\|_{L^{p_ {3}}}\|\Lambda^{s}f\|_{L^{p_{4}}}\right).
\end{align*}
\end{lemma}

\section{Global regularity and sharp decay rates of small solutions to CNS equations}  \label{section3}
In this section, we only prove the critical case with $s=\beta$ and $\beta\in(\frac 1 2,1)$. The case with $\beta=\frac{1}{2}$ is similar, so the details are left to the interested reader. For the sake of simplicity, all occurrences of $C_2$ denote any positive power of 
 $C_2$ throughout this paper. First, we prove the global regularity for system (\ref{eq1}) in critical Sobolev space. 

\textbf{Proof of Theorem \ref{1theo1}:}
\begin{proof}
 We assume that $(a,u)$ be a local strong solution of (\ref{eq1}). Before proceeding any further, we assume a
priori that
\begin{align}\label{301}
\|(a,u)\|_{H^{{\beta}}}\le \delta\ll 1.
\end{align} 
Then by Sobolev's inequality, we can deduce
\begin{align*}
\frac{1}{2}\le \rho\le2,
\end{align*}
which gives rise to
\begin{align*}
\frac{1}{4}\|u\|_{L^2}^2\le\frac{1}{2}\int_{\mathbb{R}}\rho u^2dx\le \|u\|_{L^2}^2.
\end{align*}
From the definition of $G(\rho)$ and the fact $\|a\|_{L^{\infty}}$ is small enough, one can derive 
\begin{align*}
c\|a\|_{L^{2}}^{2}\le\int_{\mathbb{R}} G(\rho)dx\le C\|a\|_{L^{2}}^{2}.
\end{align*}
Therefore, combining the above two lines, we get
\begin{align}\label{302}
c\|(a,u)\|_{L^{{2}}}^{2}\le E_{0}\le C\|(a,u)\|_{L^{{2}}}^{2}.
\end{align}
Taking the $L^2$ inner product of the first two equations of (\ref{eq0}) with $(G^{\prime}(\rho),u)$ and integrating by parts, we get
\begin{align}\label{303}
\frac{d}{dt}\int_{\mathbb{R}}\left(G(\rho)+\frac{1}{2}\rho u^2\right)dx+\|\Lambda^{\beta}u\|_{L^2}^{2}=0.
\end{align}
Applying $\Lambda^{\beta}$ to (\ref{eq1}), and taking the $L^2$ inner product of the first two equations of (\ref{eq1}) with $(\gamma\Lambda^{\beta}a,\Lambda^{\beta} u)$ and integrating by parts, by Lemma \ref{lemma6} and \ref{lemma5}, we obtain
 \begin{align}\label{304}
&\frac{1}{2}\frac{d}{dt}\|\sqrt{\gamma} \Lambda^{\beta}a\|_{L^2}^{2}+\gamma\langle\Lambda^{\beta}u_x,\Lambda^{\beta}a\rangle\\\nonumber
=&-\gamma\langle\Lambda^{\beta}(au)_{x},\Lambda^{\beta}a\rangle\\\nonumber
=&-\gamma\langle[\Lambda^{\beta},u\partial_x]a,\Lambda^{\beta}a\rangle-\gamma\langle u \Lambda^{\beta}a_x,\Lambda^{\beta}a\rangle-\gamma\langle\Lambda^{\beta}(au_{x}),\Lambda^{\beta}a\rangle\\\nonumber
=&-\gamma\langle[\Lambda^{\beta},u\partial_x]a,\Lambda^{\beta}a\rangle+\frac{\gamma}{2}\langle u_x, |\Lambda^{\beta}a|^{2}\rangle-\gamma\langle\Lambda^{2\beta-1}(au_{x}),\Lambda a\rangle\\\nonumber
\le\ &C\|\Lambda^{\beta} a\|_{L^{\frac{1}{\beta-\frac{1}{2}}}}\|\Lambda^{\beta} a\|_{L^{2}}\|u_x\|_{L^{\frac{1}{1-\beta}}}+C\|\Lambda^{2\beta}u \|_{L^{2}}\|\Lambda a\|_{L^{2}}\|a\|_{L^{\infty}}+C\|u_x\|_{L^{\frac{1}{1-\beta}}}\|\Lambda a\|_{L^{2}}\|\Lambda^{2\beta-1}a\|_{L^{\frac{1}{\beta-\frac{1}{2}}}}\\\nonumber
\le\ &C\|\Lambda a\|_{L^{2}}\|\Lambda^{\beta} a\|_{L^{2}}\|\Lambda^{\beta+\frac{1}{2}} u\|_{L^{2}}+C\|\Lambda a\|_{L^{2}}\| a\|_{L^{\infty}}\|\Lambda^{2\beta}u\|_{L^{2}}\\\nonumber
\le\ &C\delta \mathscr{D}_{0},
 \end{align}
 and
\begin{align}\label{305}
&\frac{1}{2}\frac{d}{dt}\|\Lambda^{\beta}u\|_{L^2}^{2}+\gamma\langle\Lambda^{\beta}a_x,\Lambda^{\beta}u\rangle+\|\Lambda^{2\beta}u\|_{L^2}^2\\\nonumber
=&\langle \Lambda^{\beta} (k(a)a_x),\Lambda^{\beta}u\rangle-\langle \Lambda^{\beta}(uu_x),\Lambda^{\beta}u\rangle+\left\langle\Lambda^{\beta}\left(\frac{a}{1+a}(-\Delta)^{\beta}u\right),\Lambda^{\beta} u\right\rangle\\\nonumber
\le\ & C\|\Lambda^{2\beta}u\|_{L^2}\|a_x\|_{L^2}\|a\|_{L^\infty}+C\|\Lambda^{2\beta}u\|_{L^2}\| u_x\|_{L^2}\|u\|_{L^\infty}+C\|\Lambda^{2\beta}u\|_{L^2}^{2}\|a\|_{L^\infty}\\\nonumber
\le\ & C\delta \mathscr{D}_{0}.
\end{align}

Now we use the inner product between $a$ and $u$ to generate the dissipation of $a$. we choose $k>0$, which will be determined later. 
By virtue of Sobolev's inequality,  we arrive at
\begin{align}\label{306}
&\frac{d}{dt}\langle u,k\Lambda^{-2+2 \beta} a_x\rangle+k\gamma\|\Lambda^{\beta} a\|_{L^{2}}^{2} \\ \nonumber
=&k\langle\Lambda^{2\beta-2}u_x,u_x\rangle+k\langle\Lambda^{2\beta-2}u_x,(au)_x\rangle-k\langle(-\Delta)^{\beta}u,\Lambda^{2\beta-2}a_x\rangle\\\nonumber
&+k\langle k(a)a_x,\Lambda^{2\beta-2}\nabla a\rangle+k\left\langle-uu_x+\frac{a}{1+a}(-\Delta)^{\beta}u,\Lambda^{2\beta-2}a_x\right\rangle\\\nonumber
\le\ & Ck\|\Lambda^{\beta}u\|_{L^{2}}^{2}+Ck\|\Lambda^{2\beta-1}u\|_{L^{\frac{1}{\beta-\frac{1}{2}}}}\| (a_x,u_x)\|_{L^2}\|(a,u)\|_{L^\frac{1}{1-\beta}}+Ck\|\Lambda^{3\beta-1}u\|_{L^{2}}\|\Lambda^{\beta}a\|_{L^{2}}\\\nonumber
&+Ck\|\Lambda^{2\beta-1}a\|_{L^{\frac{1}{\beta-\frac{1}{2}}}}\| (a_x,u_x)\|_{L^2}\|(a,u)\|_{L^\frac{1}{1-\beta}}+Ck\|\Lambda^{2\beta-1}a\|_{L^{\frac{1}{\beta-\frac{1}{2}}}}\|\Lambda^{2\beta} u\|_{L^2}\|a\|_{L^\frac{1}{1-\beta}}\\\nonumber
\le\ & Ck\|\Lambda^{\beta}u\|_{L^{2}}^{2}+k\|\Lambda^{\beta}(a,u)\|_{L^{2}}\| (a_x,u_x)\|_{L^2}\|\Lambda^{\beta-\frac{1}{2}}(a,u)\|_{L^2}+Ck\|\Lambda^{3\beta-1}u\|_{L^{2}}\|\Lambda^{\beta}a\|_{L^{2}}\\\nonumber
&+Ck\|\Lambda^{\beta}a\|_{L^{2}}\|\Lambda^{2\beta} u\|_{L^2}\|\Lambda^{\beta-\frac{1}{2}}a\|_{L^2}\\\nonumber
\le &C \delta \mathscr{D}_{0}+\frac{k}{100}\|\Lambda^{\beta}a\|_{L^2}^{2}+Ck\left(\|\Lambda^{3\beta-1}u\|_{L^2}^{2}+\|\Lambda^{\beta}u\|_{L^2}^{2}\right).
\end{align}
Along the same line, due to H\"{o}lder's inequality, we find
\begin{align}\label{307}
&\frac{d}{dt}\langle a_x,ku\rangle+k\gamma\|a_x\|_{L^2}^2\\\nonumber
=& k\langle u_x,u_x\rangle+k\langle(au)_{x},u_x\rangle-k\langle(-\Delta)^{\beta}u,a_x\rangle\\\nonumber
&+k\langle K(a)a_x,a_x\rangle-k\langle uu_x,a_x\rangle+k\left\langle \frac{a}{1+a}(-\Delta)^{\beta}u,a_x\right\rangle\\\nonumber
\le\ &Ck\|u_x\|_{L^2}^2+Ck\|\Lambda^{2\beta}u\|_{L^2}\|\Lambda a\|_{L^2}\\\nonumber
&+ Ck\|\Lambda (a,u)\|_{L^{2}}^{2}\|(a,u)\|_{L^{\infty}}+Ck\|\Lambda a\|_{L^{2}}\|a\|_{L^{\infty}}\|\Lambda^{2\beta}u\|_{L^{2}}\\\nonumber
\le\ &C\delta \mathscr{D}_{0}+Ck\|\Lambda^{\beta}u\|_{H^{\beta}}^{2}+\frac{k}{100}\|\Lambda a\|_{L^2}.
\end{align}
Combining (\ref{303})-(\ref{307}), we conclude that for any $k>0$
\begin{align*}
&\frac{d}{dt}\left(\int_{\mathbb{R}}\left(2G(\rho)+\rho u^2\right)dx+\|(\sqrt{\gamma} a,u)\|_{\dot{H}^{\beta}}^{2}+2k\langle \Lambda^{\beta-1}a_x,\Lambda^{\beta-1}u\rangle_{H^{1-\beta}}\right)\\
&+\left(2-Ck\right)\|\Lambda^{\beta}u\|_{H^{\beta}}^{2}+\left(2k\gamma-\frac{k}{50}\right)\|\Lambda^{\beta}a\|_{H^{1-\beta}}^{2}
\le C\delta \mathscr{D}_{0}.
\end{align*}
Choosing $k$ small enough, we obtain
\begin{align}\label{309}
\frac{d}{dt}\mathscr{E}_0+\mathscr{D}_0\le 0.
\end{align}
By $\frac{1}{2}<\beta<1$, we have 
\begin{align*}
2k\langle\Lambda^{\beta-1}a_x,\Lambda^{\beta-1}u\rangle_{H^{1-\beta}}\le &\ 2k\|a\|_{H^{\beta}}\|u\|_{H^{1-\beta}}\\ \nonumber
\le &\ 2k\|a\|_{H^{\beta}}\|u\|_{H^{{\beta}}}\\ \nonumber
\le &\ \frac{k}{2}\|a\|_{H^{\beta}}^{2}+2k\|u\|_{H^{{\beta}}}^{2}.\nonumber
\end{align*}
Taking $k$ small enough and integrating (\ref{309}) in time on $[0,t]$, then we conclude that
\begin{align*}
\sup_{t}\|(a,u)\|_{H^{\beta}}^{2}+\int_{0}^{t}\mathscr{D}_{0}(t^{\prime})dt^{\prime}\le C\|(a_0,u_0)\|_{H^{\beta}}^{2}\le \frac{\delta^2}{4}
\end{align*}
holds for small enough $\delta>0$. Then we achieve the conclusion.
\end{proof}

We are now in a position to derive the time decay rates of solutions to the CNS equations (\ref{eq1}).
\begin{prop}{\label{3prop2}}
  Under the same conditions as in Theorem \ref{1theo1}, if additionally $\left(a_{0}, u_{0}\right) \in \dot{B}_{2, \infty}^{-\frac{1}{2}}$, and there exists a small constant $\delta $ such that if 
  \begin{align}\label{310}
\|(a_{0}, u_{0})\|_{\dot{B}^{-\frac{1}{2}}_{2,\infty}}\le\delta,
  \end{align}
  then there exists $C>0$ such that for every $t>0$, there holds
\begin{align*}
\|(a,u)\|_{\dot{H}^{s_1}}\le C(1+t)^{-\frac{1+2s_1}{4\beta}},
\end{align*}
where $0\le s_1\le \beta$.
\end{prop}
\begin{proof}
In what follows, we divide the proof into the following steps:

\textbf{Step1:}
We apply the
bootstrap argument to establish the desired decay estimates. We start with the ansatz:
\begin{align}\label{3101}
(1+t)^{\frac{1}{4\beta}}\|(a,u)\|_{L^{2}}+(1+t)^{\frac{1}{4\beta}+\frac{1}{2}}\|\Lambda^{\beta}(a,u)\|_{L^{2}}\le C_{0}\delta,
\end{align}
where $C_0$ is a large constant and will be determined later. For any $\sigma<\frac{1}{2\beta}$, due to (\ref{309}), we have
\begin{align*}
(1+t)^{\sigma}\frac{d}{dt}\mathscr{E}_0+(1+t)^{\sigma}\mathscr{D}_{0}\le 0,
\end{align*}
together with (\ref{3101}) ensures that
\begin{align}\label{3102}
  (1+t)^{\sigma}\mathscr{E}_0+\int_{0}^{t}(1+t^{\prime})^{\sigma}\mathscr{D}_{0}dt^{\prime}\le C\delta^{2}+\sigma\int_{0}^{t}(1+t^{\prime})^{\sigma-1}\mathscr{E}_{0}dt^{\prime}\le CC_{0}^{2}\delta^{2}.
\end{align}
We next prove the solutions of (\ref{eq1}) belong to some negative index Besov space. Applying $\dot\Delta_{j}$ to (\ref{eq1}), we find
\begin{align}\label{3105}
 \left\{\begin{array}{l}\dot{\Delta}_{{j}}a_{t}+\dot{\Delta}_{j}u_x=\dot{\Delta}_{j}F,\\ [1ex] \dot{\Delta}_{j}u_{t}+(-\Delta)^{\beta}\dot{\Delta}_{j}u+\gamma\dot{\Delta}_{j}a_x=\dot{\Delta}_{j}H,\end{array}\right.
\end{align}
where $F=-(au)_x$ and $H=K(a)a_x-uu_x+\frac{a}{1+a}(-\Delta)^{\beta}u$. Then we get
\begin{align}\label{311}
\frac{d}{dt}\left(\gamma\|\dot{\Delta}_{j} a\|_{L^{2}}^{2}+\|\dot{\Delta}_{j}u\|_{L^{2} }^{2}\right)+2\|\Lambda^{\beta}\dot{\Delta}_{j}u\|_{L^{2}}^{2}\le C\left(\|\dot{\Delta}_{j}F\| _{L^{2}}\|\dot{\Delta}_{j}a\|_{L^{2}}+\|\dot{\Delta}_{j}H\|_{L^{2}}\|\dot{ \Delta}_{j}u\|_{L^{2}}\right).
\end{align}
Multiplying (\ref{311}) by $2^{-j}$ and taking $l^{\infty}$ norm, we have
\begin{align}\label{3111}
\frac{d}{dt}\left(\gamma\|a\|_{\dot{B}^{-\frac{1}{2}}_{2,\infty}}^{2}+\|u\|_{\dot{B}^{-\frac{1}{2}}_{2,\infty}}^{2}\right)\le C \left(\|F\|_{\dot{B}^{-\frac{1}{2}}_{2,\infty}}\|a\|_{ \dot{B}^{-\frac{1}{2}}_{2,\infty}}+\|H\|_{\dot{B}^{-\frac{1}{2}}_{2,\infty}}\|u\|_{ \dot{B}^{-\frac{1}{2}}_{2,\infty}}\right).  
\end{align}
Let $\displaystyle M(t)=\sup_{0\le t^{\prime}\le t}\left(\gamma\|a\|_{\dot{B}^{-\frac{1}{2}}_{2,\infty}}+\|u\|_{\dot{B}^{-\frac{1}{2}}_{2,\infty}}\right)$, then we yield
\begin{align*}
M^{2}(t)\leq CM^{2}(0)+CM(t)\int_{0}^{t}\left(\|F\|_{\dot{B}^{-\frac{1}{2}}_{2,\infty}}+\|H\|_{\dot{B}^{-\frac{1}{2}}_{2,\infty}}\right)dt^{\prime}. 
\end{align*}
Using the fact that $ L^{1}\hookrightarrow\dot{B}^{-\frac{1}{2}}_{2,\infty}$, (\ref{3101}) and (\ref{3102}) , we conclude, for any $t> 0$ and $1-\frac{1}{2\beta}<\sigma<\frac{1}{2\beta}$, that
\begin{align}\label{312}
\int_{0}^{t}\left(\|F\|_{\dot{B}^{-\frac{1}{2}}_{2,\infty}}+\|H\|_{\dot{B}^{-\frac{1}{2}}_{2,\infty}}\right)dt^{\prime}&\le C\int_{0}^{t}\left(\|F\|_{L^1}+\|H\|_{L^1}\right)dt^{\prime}\\\nonumber
&\le   CC_{0}\delta\int_{0}^{t}(1+t^{\prime})^{-\frac{1}{4\beta}}\mathscr{D}_{0}^{\frac{1}{2}}dt^{\prime}\\\nonumber
&\le   CC_{0}\delta\left(\int_{0}^{t}(1+t^{\prime})^{-\frac{1}{2\beta}-\sigma}dt^{\prime}\right)^{\frac{1}{2}}\left(\int_{0}^{t}(1+t^{\prime})^{\sigma}\mathscr{D}_{0}dt^{\prime}\right)^{\frac{1}{2}}\le CC_{0}^{2}\delta^{2},
\end{align}
which together with (\ref{310}) ensures that $M(t)\le CC_{0}^{2}\delta^{2}+C\delta$. 

\textbf{Step2:} Denote $S(t)=\left\{\xi:|\xi|^{2\beta}\le C_{2}(1+t)^{-1}\right\}$, $C_{2}$ is large enough.  Due to (\ref{309}), we infer that
  \begin{align}\label{313}
  \frac{d}{dt}E_{0}(t)+\frac{C_{2}}{1+t}\left(k\gamma\|a\|_{H^{\beta}}^{2}+\|u\|_{H^{\beta}}^{2}\right)\leq \frac{C}{1+t}{\int_{S(t)}|\widehat{a}(\xi)|^{2}+|\widehat u}(\xi)|^{2}d\xi.
  \end{align}  
  Applying
Fourier transform to (\ref{eq1}), we have
\begin{align*}
\left\{\begin{array}{l}\widehat{a}_{t}-i\xi\widehat{u}=\widehat{F},\\ [1ex]
 \widehat{u}_{t}+|\xi|^{2\beta}\widehat{u}+i\gamma\xi\widehat{a}=\widehat{H}.\end{array}\right.
\end{align*}
Then we deduce
\begin{align}\label{314}
\frac{1}{2}\frac{d}{dt}\left(\gamma|\widehat{a}|^{2}+|\widehat{u}|^{2}\right)+|\xi|^{2\beta}|\widehat{u}|^{2}=\gamma\mathcal{R}e[\widehat{F} \cdot\overline{\widehat{a}}]+\mathcal{R}e[\widehat{H}\cdot\overline{\widehat{u}}]
\end{align}
Integrating (\ref{314}) in time on $[0,t]$, we get
\begin{align}\label{315}
\gamma|\widehat{a}|^{2}+|\widehat{u}|^{2}\le C\left(|\widehat{a}_{0}|^{2}+|\widehat{u}_{0}|^{2}\right)+C\int_{0}^{t}|\widehat{F}\cdot\overline{\widehat{a}}|+|\widehat{H}\cdot\overline{\widehat{u}}|dt^{\prime}.
\end{align}
Due to the fact $E(0)<\infty$ and (\ref{310}), we find
\begin{align}\label{316}
\int_{S(t)}\left(|\widehat{a_{0}}|^{2}+|\widehat{u_{0}}|^{2}\right)d\xi &\le\sum_{j\leq\log_{2}\left[\frac{4}{3}C_{2}^{\frac{1}{2\beta}}(1+t)^{-\frac{1}{2\beta}}\right]}\int_{\mathbb{R}^{2}}2\varphi^{2}(2^{-j}\xi)\left(|\widehat{a_{0}}|^{2}+|\widehat{u_{0}}|^{2}\right)d\xi \\ \nonumber
 & \le\sum_{j\leq\log_{2}\left[\frac{4}{3}C_{2}^{\frac{1}{2\beta}}(1+t)^{-\frac{1}{2\beta}}\right]}2\left(\|\dot{\Delta}_{j}a_{0}\|_{L^{2}}^{2}+\|\dot{\Delta}_{j}u_{0}\|_{L^{2}}^{2}\right) \\ \nonumber
 & \le C\delta^{2}(1+t)^{-\frac{1}{2\beta}}.\nonumber
\end{align}
By (\ref{3101}), we arrive at
\begin{align}\label{317}
\int_{S(t)}\int_{0}^{t}|\widehat{F}\cdot\overline{\widehat{a}}|+|\widehat{H}\cdot\overline{\widehat{u}}|dt^{\prime}d\xi=& \int_{0}^{t}\int_{S(t)}|\widehat{F}\cdot\overline{\widehat{a}}|+|\widehat{H}\cdot\overline{\widehat{u}}|d\xi dt^{\prime} \\\nonumber
\le & C\int_{0}^{t}\left(\|F\|_{L^1}\int_{S(t)}|{\widehat{a}}|d\xi+\|H\|_{L^1}\int_{S(t)}|{\widehat{u}}|d\xi\right
) dt^{\prime} \\\nonumber
\le & \ C(1+t)^{-\frac{1}{4\beta}}\int_{0}^{t}\left[\left(\|F\|_{L^1}+\|H\|_{L^1}\right)\left(\int_{S(t)}|\widehat{a}|^{2}+|\widehat{u}|^{2}d\xi\right)^{\frac{1}{2}}\right]dt^{\prime}\\\nonumber
\le& CC_{0}^{2}\delta^{3}(C_{0}^{2}\delta+1)(1+t)^{-\frac{1}{2\beta}},\nonumber
\end{align}
where we use the fact
\begin{align*}
\left(\int_{S(t)}|\widehat{a}|^{2}+|\widehat{u}|^{2}d\xi\right)^{\frac{1}{2}}dt^{\prime}\le C(C_{0}^{2}\delta+1)(1+t)^{-\frac{1}{4\beta}}.
\end{align*}
Plugging (\ref{316}) and (\ref{317}) into (\ref{313}) gives
\begin{align*}
 \frac{d}{dt}\mathscr{E}_{0}(t)+\frac{C_{2}}{1+t}\left(k\gamma\|a\|_{H^{\beta}}^{2}+\|u\|_{H^{\beta}}^{2}\right)\leq C\left[C_{0}^{2}\delta^{3}(C_{0}^{2}\delta+1)+\delta^{2}\right](1+t)^{-\frac{1}{2\beta}-1},
\end{align*}
so that there holds
\begin{align}\label{318}
 \mathscr{E}_{0}(t)\le  C\left[C_{0}^{2}\delta^{3}(C_{0}^{2}\delta+1)+\delta^{2}\right](1+t)^{-\frac{1}{2\beta}}.
  \end{align}
  
\textbf{Step3:}
  By (\ref{318}), we get
\begin{align}\label{3181}
  \frac{d}{dt}\left[(1+t)^{\frac{1}{2\beta}+1}\mathscr{E}_{0}\right]+(1+t)^{\frac{1}{2\beta}+1}\mathscr{D}_{0}\le C\left[C_{0}^{2}\delta^{3}(C_{0}^{2}\delta+1)+\delta^{2}\right](1+t)^{\frac{1}{2\beta}}\mathscr{E}_{0}.
  \end{align}
  Integrating (\ref{3181}) over $[0,t]$, we find
    \begin{align*}
(1+t)^{\frac{1}{2\beta}+1}\mathscr{E}_{0}+\int_{0}^{t}(1+t^{\prime})^{\frac{1}{2\beta}+1}\mathscr{D}_{0}dt^{\prime}\le C\left[C_{0}^{2}\delta^{3}(C_{0}^{2}\delta+1)+\delta^{2}\right](1+t),
  \end{align*}
  which ensures that
  \begin{align}\label{322}
  \frac{1}{1+t}\int_{0}^{t}(1+t^{\prime})^{\frac{1}{2\beta}+1} \mathscr{D}_{0}dt^{\prime}\le C\left[C_{0}^{2}\delta^{3}(C_{0}^{2}\delta+1)+\delta^{2}\right].
  \end{align}
  Thanks to Lemma \ref{lemma3}, (\ref{304}) and (\ref{3101}), we have
   \begin{align}\label{323}
&\frac{1}{2}\frac{d}{dt}\|\sqrt{\gamma} \Lambda^{\beta}a\|_{L^2}^{2}+\gamma\langle\Lambda^{\beta}u_x,\Lambda^{\beta}a\rangle\\\nonumber
\le\ &C\|\Lambda a\|_{L^{2}}\|\Lambda^{\beta} a\|_{L^{2}}\|\Lambda^{\beta+\frac{1}{2}} u\|_{L^{2}}+C\|\Lambda a\|_{L^{2}}\| a\|_{L^{\infty}}\|\Lambda^{2\beta}u\|_{L^{2}}\\\nonumber
\le\ &C\|\Lambda a\|_{L^{2}}\|\Lambda^{\beta} a\|_{L^{2}}\|u\|_{L^{2}}^{\frac{\beta-\frac{1}{2}}{2\beta}}\|\Lambda^{2\beta} u\|_{L^{2}}^{\frac{\beta+\frac{1}{2}}{2\beta}}+C\|\Lambda a\|_{L^{2}}\| a\|_{L^{\infty}}\|\Lambda^{2\beta}u\|_{L^{2}}\\\nonumber
\le\ &\frac{1}{100}\|\Lambda^{2\beta} u\|_{L^{2}}^{2}+C\|\Lambda a\|_{L^{2}}^{\frac{4\beta}{3\beta-\frac{1}{2}}}\|\Lambda^{\beta} a\|_{L^{2}}^{\frac{4\beta}{3\beta-\frac{1}{2}}}\|u\|_{L^{2}}^{\frac{2\beta-1}{3\beta-\frac{1}{2}}}+C\|\Lambda a\|_{L^{2}}^{2}\| a\|_{L^{\infty}}^{2}\\\nonumber
\le\ &\frac{1}{100}\|\Lambda^{2\beta} u\|_{L^{2}}^{2}+C\left(\|\Lambda^{\beta} a\|_{L^{2}}^{\frac{2\beta+1}{3\beta-\frac{1}{2}}}\|u\|_{L^{2}}^{\frac{2\beta-1}{3\beta-\frac{1}{2}}}+\|a\|_{L^{\infty}}^{2}\right) \mathscr{D}_{0}\\\nonumber
\le\ &\frac{1}{100}\|\Lambda^{2\beta} u\|_{L^{2}}^{2}+CC_{0}^{2}\delta^{\frac{8\beta}{6\beta-1}}(1+t)^{-\frac{2\beta+3}{6\beta-1}}\mathscr{D}_{0},\nonumber
 \end{align}
 and
 \begin{align}\label{324}
&\frac{1}{2}\frac{d}{dt}\|\Lambda^{\beta}u\|_{L^2}^{2}+\gamma\langle\Lambda^{\beta}a_x,\Lambda^{\beta}u\rangle+\|\Lambda^{2\beta}u\|_{L^2}^2\\\nonumber
=&\langle \Lambda^{\beta} (k(a)a_x),\Lambda^{\beta}u\rangle-\langle \Lambda^{\beta}(u
 u_x),\Lambda^{\beta}u\rangle+\left\langle\Lambda^{\beta}\left(\frac{a}{1+a}(-\Delta)^{\beta}u\right),\Lambda^{\beta} u\right\rangle\\\nonumber
\le\ & C\|\Lambda^{2\beta}u\|_{L^2}\|a_x\|_{L^2}\|a\|_{L^\infty}+C\|\Lambda^{2\beta}u\|_{L^2}\| u_x\|_{L^2}\|u\|_{L^\infty}+C\|\Lambda^{2\beta}u\|_{L^2}^{2}\|a\|_{L^\infty}\\\nonumber
\le & \left(\frac{1}{100}+\delta\right)\|\Lambda^{2\beta}u \|_{L^{2}}^{2}+C\|(a_x,u_x)\|_{L^2}^{2}\|(a,u)\|_{L^\infty}^{2}.
\end{align}
Combining (\ref{323}) and (\ref{324}), we deduce
 \begin{align}\label{325}
\frac{d}{dt}\| \Lambda^{\beta}(\sqrt{\gamma}a,u)\|_{L^2}^{2}+\|\Lambda^{2\beta}u\|_{L^2}^2\le CC_{0}^{2}\delta^{\frac{8\beta}{6\beta-1}}(1+t)^{-\frac{2\beta+3}{6\beta-1}}\mathscr{D}_{0}.
 \end{align}
Multiplying (\ref{325}) by $(1+t)^{\frac{1}{2\beta}+2}$, one can conclude that
\begin{align}\label{326}
&\frac{d}{dt}\left[(1+t)^{\frac{1}{2\beta}+2}\| \Lambda^{\beta}(\sqrt{\gamma}a,u)\|_{L^2}^{2}\right]+(1+t)^{\frac{1}{2\beta}+2}\|\Lambda^{2\beta}u\|_{L^2}^2\\\nonumber
\le &CC_{0}^{2}\delta^{\frac{8\beta}{6\beta-1}}(1+t)^{\frac{1}{2\beta}+1}\mathscr{D}_{0}+\left(2+\frac{1}{2\beta}\right)(1+t)^{\frac{1}{2\beta}+1}\| \Lambda^{\beta}(\sqrt{\gamma}a,u)\|_{L^2}^{2}.
\end{align}
Integrating (\ref{326}) over $[0,t]$ and using (\ref{322}) gives rise to
\begin{align}\label{327}
&(1+t)^{\frac{1}{2\beta}+2}\| \Lambda^{\beta}(\sqrt{\gamma}a,u)\|_{L^2}^{2}\\\nonumber
\le &CC_{0}^{2}\delta^{\frac{8\beta}{6\beta-1}}\int_{0}^{t}(1+t^{\prime})^{\frac{1}{2\beta}+1}\mathscr{D}_{0}dt^{\prime}+\left(2+\frac{1}{2\beta}\right)\int_{0}^{t}(1+t^{\prime})^{\frac{1}{2\beta}+1}\| \Lambda^{\beta}(\sqrt{\gamma}a,u)\|_{L^2}^{2}dt^{\prime}\\\nonumber
\le &C\left[C_{0}^{2}\delta^{3}(C_{0}^{2}\delta+1)^{2}+\delta^{2}\right](1+t),
\end{align}
Consequently, we obtain
\begin{align}\label{328}
\|\Lambda^{\beta}(a,u)\|_{L^2}^{2}\le  C\delta^{2}\left[C_{0}^{2}\delta(C_{0}^{2}\delta+1)^{2}+1\right](1+t)^{-\frac{1}{2\beta}-1}.
\end{align}
Now, adding (\ref{318}) and (\ref{328}),  it holds that
\begin{align}\label{329}
(1+t)^{\frac{1}{4\beta}}\|(a,u)\|_{L^{2}}+(1+t)^{\frac{1}{4\beta}+\frac{1}{2}}\|\Lambda^{\beta}(a,u)\|_{L^{2}}\le CC_{0}^{3}\delta^{\frac{3}{2}}+C\delta.
\end{align}
If we choose $\delta$ and $C_0$ satisfying $4C< C_{0}$ and $C_{0}^{3}\delta^{\frac{1}{2}}<1$. Then by (\ref{329}), we have
\begin{align}\label{330}
(1+t)^{\frac{1}{4\beta}}\|(a,u)\|_{L^{2}}+(1+t)^{\frac{1}{4\beta}+\frac{1}{2}}\|\Lambda^{\beta}(a,u)\|_{L^{2}}<\frac{C_{0}\delta}{2}.
\end{align}
Therefore, we close the continuity argument.  From this, we obtain
\begin{align*}
\|\Lambda^{s_1}(a,u)\|_{{L}^{2}}\le C(1+t)^{-\frac{1+2s_1}{4\beta}},
\end{align*}
where $0\le s_1\le \beta$.
This completes the proof of Proposition \ref{3prop2}.
 \end{proof}
 \textbf{Proof of Theorem \ref{1theo2}:}
\begin{proof} 
  Next we want to prove the lower bound to the solutions for (\ref{eq1}).

 For the sake of convenience, we agree that occurrences of $\delta$ denote its positive powers throughout this proof. We first consider the linear system:
\begin{align}\label{332}
\left\{\begin{array}{l}\partial_{t}a_{L}+ \partial_{x}{u_{L}}=0,\ \\ [1ex] \partial_{t}u_{L}+\gamma\partial_{x}a_{L}+(-\Delta)^{\beta}u_{L}=0,\\[1ex]
  a_{L}|_{t=0}=a_{0},\ u_{L}|_{t=0}=u_{0}.\end{array}\right.
\end{align}
According to Proposition \ref{3prop2}, one can deduce that there holds $\|\Lambda^{s_1}(a_{L},u_{L})\|_{L^2}\le C(1+t)^{-\frac{1+2s_1}{4\beta}}$ with $0\le s_1\le s+\beta$ and $(a_{L},u_{L})\in L^{ \infty}([0,\infty),\dot{B}_{2,\infty}^{-\frac{1}{2}})$. Applying Fourier transform to (\ref{332}), we get
\begin{align}\label{333}
\left\{\begin{array}{l}\partial_{t}\widehat{a}_{L}-i\xi\widehat{u}_{L}=0,\\ [1ex] \partial_{t}\widehat{u}_{L}+|\xi|^{2\beta}\widehat{u}_{L}+\gamma i\xi\widehat{a}_{L}=0.\end{array}\right.
\end{align}
Then we observe from (\ref{333}) that
\begin{align*}
\frac{1}{2}\frac{d}{dt}\left[e^{2|\xi|^{2\beta}t}|(\sqrt{\gamma}\widehat{a}_{L},\widehat{u}_{L})|^{2}\right]-\gamma| \xi|^{2\beta}e^{2|\xi|^{2\beta}t}|\widehat{a}_{L}|^{2}=0,
\end{align*}
which implies that
\begin{align*}
|\xi|^{2s_{1}}|(\sqrt{\gamma}\widehat{a}_{L},\widehat{u}_{L})|^{2}=|\xi|^{2s_{1}}e^{-2|\xi|^{2\beta}t }|(\sqrt{\gamma}\widehat{a}_{0},\widehat{u}_{0})|^{2}+\int_{0}^{t}2\gamma|\xi|^{2(s_{1}+\beta)}e^{-2|\xi|^{2\beta}(t-t^{\prime})}|\widehat{a}_{L}|^{2}dt^{\prime}.
\end{align*}
Due to the fact $0<c_{0}=|\int_{\mathbb{R}}(a_{0}, u_{0})dx|=|(\widehat{a}_{0}(0),\widehat{u}_{0}(0))|$, we deduce that there exists $\eta>0$ such that $|(\widehat{a}_{0}(\xi),\widehat{u}_{0}(\xi))|\geq\frac{c_{0}}{2}$ if $\xi\in B(0, \eta)$. Then we have 
 \begin{align}\label{334}
   \|(\sqrt{\gamma}a_{L},u_{L})\|_{\dot{H}^{s_{1}}}^{2}& \geq\int_{|\xi|\leq\eta}|\xi|^{2s_{1}}e^{-2|\xi|^{2\beta}t}|(\sqrt{\gamma}\widehat{a}_{0}, \widehat{u}_{0})|^{2}d\xi\\ \nonumber &\geq\frac{c_{0}^{2}}{4}\int_{|\xi|\leq\eta}|\xi|^{2s_{1}}e^{-2|\xi|^{2\beta}t}d\xi\\ \nonumber
   &\geq C_{\beta}^{2}(1+t)^{-\frac{1+2s_{1}}{2\beta}},\nonumber
 \end{align}
 where $C_\beta^2=\frac{c_{0}^{2}}{4}\int_{|y|\le \eta}|y|^{2s_{1}}e^{-2|y|^{2\beta}}dy$.
 
  Taking $a_{N}=a-a_{L}$ and $u_{N}= u-u_{L}$. By Proposition \ref{3prop2}, we easily deduce $\|\Lambda^{s_{1}}(a_{N},u_{N})\|_{L^{2}}^{2} \leq C(1+t)^{-\frac{2s_{1}+1}{2\beta}}$ and $(a_{N},u_{N})\in L^{ \infty}([0,\infty),\dot{B}_{2,\infty}^{-\frac{1}{2}})$.  Moreover, we have
\begin{align}\label{335}
\left\{\begin{array}{l}\partial_{t}a_{N}+ \partial_{x}{u_{N}}=F,\ \\ [1ex] \partial_{t}u_{N}+\gamma\partial_{x}a_{N}+(-\Delta)^{\beta}u_{N}=H,\\[1ex]
  a_{N}|_{t=0}=0,\ u_{N}|_{t=0}=0.\end{array}\right.
\end{align}
 According to the time decay rates for $(a_N,u_N)$ and $(a,u)$, we conclude from (\ref{335}) that
\begin{align*}
\frac{1}{2}\frac{d}{dt}\|(\sqrt{\gamma}a_{N},u_{N})\|_{L^{2}}^ {2}+\|\Lambda^{\beta}u_{N}\|_{L^{2}}^{2}&=\gamma\langle F,a_{N}\rangle+ \langle H,u_{N}\rangle\\\nonumber
 &\le C\|(a_{N},u_ {N})\|_{L^{2}}\|(a_x,u_x)\|_{L^{2}}\|(a,u)\|_{L^{\infty}}+C\|a\|_{L^{\infty}}\|u_{N}\|_{L^{2}}\|\Lambda^{2\beta}u\|_{L^{2}}\\ \nonumber
&\le C\delta(1+t)^{-\frac{1}{2\beta}-1},\\\nonumber
\end{align*}
and
\begin{align*}
\frac{d}{dt}\langle\Lambda^{2\beta-2}u_{N},\partial_{x}a_{N}\rangle+ \gamma\|\Lambda^{\beta} a_{N}\|_{L^{2}}^{2}=&\langle\Lambda^{2\beta-2}( H-(-\Delta)^{\beta}u_{N}),\partial_{x}a_{N}\rangle\\\nonumber
&-\langle\Lambda^{2\beta-2}\ (F-\partial_{x}u_{N}),\partial_{x}u_{N}\rangle\\\nonumber
\le &C(1+t)^{-\frac{1}{2\beta}-1}.\nonumber
\end{align*}
Then we conclude that
\begin{align}\label{336}
&\frac{d}{dt}\left[\|(\sqrt{\gamma}a_{N},u_{N})\|_{L^{2}}^{2}+k_{0}\langle\Lambda^{2\beta-2}u_{N},\partial_{x}a_{N}\rangle\right]+\frac{ k_{0} C_{2}}{1+t}\|a_{N}\|_{L^{2}}^{2}+\frac{2C_{2}}{1+t}\|u_{N}\|_{L^{2}}^{2} \\\nonumber
 \le &\frac{CC_{2}}{1+t}\int_{S(t)}|\widehat{a_{N}}(\xi)|^{2}+|\widehat{ u_{N}}(\xi)|^{2}d\xi+C(\delta+k_{0})(1+t)^{-\frac{1}{\beta}-1},\nonumber
\end{align}
where $k_0$ is a small positive number, which will be determined later. We deduce from a similar derivation of (\ref{317}) that
\begin{align}\label{337}
\int_{S(t)}|\widehat{a_{N}}|^{2}+|\widehat{u_{N}}|^{2}d \xi&\le C\int_{S(t)}\int_{0}^{t}|\widehat{F}\cdot\overline{\widehat{a_ {N}}}|+|\widehat{H}\cdot\overline{\widehat{u_{N}}}|dt^{\prime}d\xi\\ \nonumber
&\le  C(1+t)^{-\frac{1}{2\beta}}\int_{0}^{t}\|(a,u)\|_{L^{2}}\left(\|\partial_{x}(a,u)\|_{L^{2}}+\|\Lambda^{2\beta} u\|_{L^{2}}\right)\|(a_{N},u_{N})\|_{\dot{B}^{-\frac{1}{2}}_{2,\infty}}dt^{\prime}\\ &\le C\delta(1+t)^{-\frac{1}{2\beta}}.\nonumber
\end{align}
Plugging (\ref{337}) into (\ref{336}), we find
\begin{align}\label{338}
&\frac{d}{dt}\left[\|(\sqrt{\gamma}a_{N},u_{N})\|_{L^{2}}^{2}+k_{0}\langle\Lambda^{2\beta-2}u_{N},\partial_{x} a_{N}\rangle\right]+\frac{ k_{0}C_{2}}{1+t}\|a_{N}\|_{L^{2}}^{2}+\frac{2C_{2}}{1+t}\|u_{N}\|_{L^{2}}^{2} \\ \nonumber
\le& C\left(\delta C_2+k_{0}\right)(1+t)^{-\frac{1}{\beta}-1}.\nonumber
\end{align}
Thanks to Proposition \ref{3prop2}, integrating (\ref{338}) over $[0,t]$, we arrive at
\begin{align*}
&(1+t)^{1+\frac{1}{2\beta}}\|(\sqrt{\gamma}a_{N},u_{N})\|_{L^{2}}^{2}\\\nonumber
\le &\ C\int_{0}^{t}2k_{0}(1+t^{\prime})^{\frac{1}{2\beta}}|\langle\Lambda^{2\beta-2}u_{N},\partial_{x}a_{N}\rangle| dt^{\prime}
+2Ck_{0}(1+t)^{1+\frac{1}{2\beta}}|\langle\Lambda^{2\beta-2}u_{N},\partial_{x}a_{N}\rangle|+C(\delta C_2+k_{0})(1+t)\\\nonumber
 \le &\ C\left(\delta C_2+k_{0}\right)(1+t),
\end{align*}
which yields
\begin{align}\label{339}
\|(\sqrt{\gamma}a_{N},u_{N})\|_{L^{2}}^{2}\le C(\delta C_2+k_{0})(1+t)^{-\frac{1}{2\beta}}.
\end{align}
Applying $\Lambda^{s_{1}}$ to (\ref{335}), we get
\begin{align}\label{3391}
\left\{\begin{array}{l}\partial_{t}\Lambda^{s_{1}}a_{N}+ \Lambda^{s_{1}}\partial_{x}{u_{N}}=\Lambda^{s_{1}}F,\ \\ [1ex] \partial_{t}\Lambda^{s_{1}}u_{N}+\gamma\Lambda^{s_{1}}\partial_{x}a_{N}+(-\Delta)^{\beta}\Lambda^{s_{1}}u_{N}=\Lambda^{s_{1}}H.\end{array}\right.
\end{align}
Let $b=2-\frac{1}{\beta}$. Standard energy estimate yields 
 \begin{align}\label{3392}
&\frac{1}{2} \frac{d}{d t}\left[(1+t)^{b}\|\Lambda^{\beta}(\sqrt{\gamma}a_N, u_N)\|_{L^{2}}^{2}\right]+(1+t)^{b}\|\Lambda^{2\beta} u_N\|_{L^{2}}^{2}\\\nonumber
=&b(1+t)^{b-1}\|\Lambda^{\beta}(\sqrt{\gamma}a_N, u_N)\|_{L^{2}}^{2}+(1+t)^{b}\left\langle\gamma\Lambda^{\beta} F, \Lambda^{\beta} a_N\right\rangle+(1+t)^{b}\left\langle\Lambda^{\beta} H, \Lambda^{\beta} u_N\right\rangle.\nonumber
\end{align}
From (\ref{335}), we know
   \begin{align}\label{340}
&(1+t)^{b}\left\langle\gamma\Lambda^{\beta} F, \Lambda^{\beta} a_N\right\rangle\\\nonumber
\le &C(1+t)^{b}(\|\Lambda a_{N}\|_{L^{2}}\|\Lambda^{\beta} a\|_{L^{2}}\|\Lambda^{\beta+\frac{1}{2}} u\|_{L^{2}}+\|\Lambda a_{N}\|_{L^{2}}\| a\|_{L^{\infty}}\|\Lambda^{2\beta}u\|_{L^{2}})\\\nonumber
\le &C\delta(1+t)^{-\frac{3}{2\beta}},
 \end{align}
 and
 \begin{align}\label{341}
&(1+t)^{b}\left\langle\Lambda^{\beta} H, \Lambda^{\beta} u_N\right\rangle\\\nonumber
\le& C(1+t)^{b}(\|\Lambda^{2\beta}u_N\|_{L^2}\|a_x\|_{L^2}\|a\|_{L^\infty}+\|\Lambda^{2\beta}u_N\|_{L^2}\| u_x\|_{L^2}\|u\|_{L^\infty}+\|\Lambda^{2\beta}u_N\|_{L^2}\|\Lambda^{2\beta}u\|_{L^{2}}\|a\|_{L^\infty})\\\nonumber
\le& C\delta(1+t)^{-\frac{3}{2\beta}}.
\end{align}
 Plugging (\ref{340}) and (\ref{341}) into (\ref{3392}), we have
  \begin{align}\label{342}
&\frac{1}{2} \frac{d}{d t}\left[(1+t)^{b}\|\Lambda^{\beta}(\sqrt{\gamma}a_N, u_N)\|_{L^{2}}^{2}\right]+(1+t)^{b}\|\Lambda^{2\beta} u_N\|_{L^{2}}^{2}\\\nonumber
\le &\ b(1+t)^{b-1}\|\Lambda^{\beta}(\sqrt{\gamma}a_N, u_N)\|_{L^{2}}^{2}+C\delta(1+t)^{-\frac{3}{2\beta}}.
\end{align}
Due to (\ref{307}), we have
\begin{align}\label{343}
&\frac{d}{dt}\langle\partial_x a_N,k_0u_N\rangle+k_0\gamma\|\partial_x a_N\|_{L^2}^2\\\nonumber
=& k_0\langle u_x,\partial_x u_N\rangle+k_0\langle(au)_{x},\partial_x u_N\rangle-k_0\langle(-\Delta)^{\beta}u,\partial_x a_N\rangle\\\nonumber
&+k_0\langle K(a)a_x,\partial_x a_N\rangle-k_0\langle uu_x,\partial_x a_N\rangle+k_0\left\langle \frac{a}{1+a}(-\Delta)^{\beta}u,\partial_x a_N\right\rangle\\\nonumber
\le\ &Ck_0\|\partial_x u\|_{L^2}\|\partial_x u_N\|_{L^2}+Ck_0\|\Lambda^{2\beta}u\|_{L^2}\|\Lambda a_N\|_{L^2}\\\nonumber
&+ Ck_0\|\Lambda (a,u)\|_{L^{2}}\|\Lambda (a_N,u_N)\|_{L^{2}}\|(a,u)\|_{L^{\infty}}+Ck_0\|\Lambda a_N\|_{L^{2}}\|a\|_{L^{\infty}}\|\Lambda^{2\beta}u\|_{L^{2}}\\\nonumber
\le\ &C(k_0+\delta)(1+t)^{-\frac{3}{2\beta}},
\end{align}
Combining (\ref{342}) to (\ref{343}), we arrive at
\begin{align}\label{344}
&\frac{d}{d t}\left[(1+t)^{b}\|\Lambda^{\beta}(\sqrt{\gamma}a_N, u_N)\|_{L^{2}}^{2}+\langle\partial_x a_N,2k_0u_N\rangle\right]+2(1+t)^{b}\|\Lambda^{2\beta} u_N\|_{L^{2}}^{2}+2k_0\gamma\|\partial_x a_N\|_{L^2}^2\\\nonumber
\le &C(k_0+\delta)(1+t)^{-\frac{3}{2\beta}}+b(1+t)^{b-1}\|\Lambda^{\beta}(\sqrt{\gamma}a_N, u_N)\|_{L^{2}}^{2}.
\end{align}
Recall the definition of $S(t)$ and taking $k_{0}C_{2}$ big enough, by (\ref{339}), we have
\begin{align}\label{345}
&\frac{d}{d t}\underbrace{\left[(1+t)^{b}\|\Lambda^{\beta}(\sqrt{\gamma}a_N, u_N)\|_{L^{2}}^{2}+\langle\partial_x a_N,2k_0u_N\rangle\right]}_{\tilde{E}(t)}\\\nonumber
&+\underbrace{C_{2}(1+t)^{b-1}\|\Lambda^{\beta} u_N\|_{L^{2}}^{2}+k_0C_2(1+t)^{b-1}\|\Lambda^{\beta} a_N\|_{L^2}^2}_{\tilde{D}(t)}\\\nonumber
\le &C(k_0+\delta)(1+t)^{-\frac{3}{2\beta}}+C(1+t)^{b-1}\int_{S(t)}|\xi|^{2\beta}\left(|\widehat{a_{N}}(\xi)|^{2}+|\widehat{u_{N}}(\xi)|^{2}\right)d \xi\\\nonumber
\le &C(\delta C_2+k_{0}+\delta)(1+t)^{-\frac{3}{2\beta}},
\end{align}
Multiplying $(1+t)^{\frac{3}{2\beta}}$ to (\ref{345}) and integrating in $[0,t]$, by (\ref{339}) again, it follows that
\begin{align}\label{346}
(1+t)^{\frac{3}{2\beta}}\tilde{E}\le & C\int_{0}^{t}(1+t^{\prime})^{\frac{3}{2\beta}-1}\langle\partial_x a_N,2k_0u_N\rangle dt^{\prime}+C(\delta C_2+k_{0}+\delta)(1+t)\\\nonumber
\le & Ck_0(\delta C_2+k_{0})^{\frac{1}{2}}\int_{0}^{t}(1+t^{\prime})^{\frac{1}{2\beta}-1} dt^{\prime}+C(\delta C_2+k_{0}+\delta)(1+t)\\\nonumber
\le & C\left[k_0(\delta C_2+k_{0})^{\frac{1}{2}}+(\delta C_2+k_{0}+\delta)\right](1+t).\nonumber
\end{align}
Moreover, we deduce that
\begin{align}\label{347}
&(1+t)^{2+\frac{1}{2\beta}}\|\Lambda^{\beta}(\sqrt{\gamma}a_{N},u_{N})\|_{L^{2}}^ {2}\\\nonumber
\le &C\left[k_0(\delta C_2+k_{0})^{\frac{1}{2}}+(\delta C_2+k_{0}+\delta)\right](1+t)-(1+t)^{\frac{3}{2\beta}}\langle\partial_x a_N,2k_0u_N\rangle\\\nonumber
\le &C\left[k_0(\delta C_2+k_{0})^{\frac{1}{2}}+(\delta C_2+k_{0}+\delta)\right](1+t),\nonumber
\end{align}
which implies that
\begin{align}\label{348}
\|\Lambda^{\beta}(\sqrt{\gamma}a_{N},u_{N})\|_{L^{2}}^ {2}
\le C\left[k_0(\delta C_2+k_{0})^{\frac{1}{2}}+(\delta C_2+k_{0}+\delta)\right](1+t)^{-\frac{1}{2\beta}-1}.
\end{align}
Taking $k_0(\delta C_2+k_{0})^{\frac{1}{2}}+(\delta C_2+k_{0}+\delta)$ small enough leads to 
\begin{align}\label{349}
\|\Lambda^{\beta}(\sqrt{\gamma}a_{N},u_{N})\|_{L^{2}}^ {2}\le \frac{C_{\beta}^2}{4}(1+t)^{-\frac{1}{2\beta}-1}.
\end{align}
Combining (\ref{339}) and (\ref{343}), we get for any $0\le s_1\le\beta$, that
\begin{align}\label{3444}
\|\Lambda^{s_1}(\sqrt{\gamma}a_{N},u_{N})\|_{L^{2}}^ {2}\le \frac{C_{\beta}^2}{4}(1+t)^{-\frac{1+2s_1}{4\beta}}.
\end{align}
Thanks to (\ref{334}) and (\ref{344}), we have
\begin{align}\label{3454}
\|\Lambda^{s_1}(\sqrt{\gamma}a,u)\|_{L^{2}}^ {2}\ge \frac{C_{\beta}^2}{4}(1+t)^{-\frac{1+2s_1}{4\beta}}.
\end{align}
This completes the proof of Theorem \ref{1theo2}.
 \end{proof}

\section{Global regularity and sharp decay rates of large solutions to CNS equations}  \label{section4}
In this section, we firstly consider the global regularity for the system (\ref{eq1}) with large initial data.  Then we consider the time decay rates for the large solutions. As introduced in Section \ref{section1}, we first get  $H^1$ time decay rate  by virtue of the method from \cite{Guo01012012}. Then we can apply Schonbek's strategy to obtain the initial $L^2$ time decay rate.  In this section, we restrict our attention here to the case with $s=1+\beta$ and $\frac{1}{2}<\beta<\frac{3}{4}$. The proofs for the case $s=\frac{1}{2}$ can be treated in a same way, it thus is omitted here.
 Now we are going to present the proof of Theorem \ref{1theo3}.
 
\textbf{Proof of Theorem \ref{1theo3}}:

\begin{proof} Before proceeding any further, we assume a priori that
\begin{align*}
\|(a,u)\|_{L^{2}}\leq M,\quad\|\partial_{x}(a,u)\|_{H^{\beta}}\leq\delta(M),
\end{align*}
where $\delta(M)$ is a small enough positive constant depending on $M$.
From (\ref{302}) and (\ref{303}), we deduce that
\begin{align}\label{403}
\|(a,u)\|_{L^2}^{2}\le C\|(a_0,u_0)\|_{L^2}^{2}\le\frac{M^2}{4}
\end{align}
holds if we taking $\|(a_0,u_0)\|_{L^2}\le\frac{M}{2}$.
Taking the $L^2$ inner product of the first two equations of (\ref{eq1}) with $(\sqrt{\gamma}\Lambda a,\Lambda u)$. Using Lemma \ref{lemma3}, \ref{lemma5} and taking $\delta$ small enough, we get
\begin{align}\label{404}
\frac{1}{2}\frac{d}{dt}\|\sqrt{\gamma} \Lambda a\|_{L^2}^{2}+\gamma\langle\Lambda u_x,\Lambda a\rangle=&-\gamma\langle\Lambda(au)_x,\Lambda a\rangle\\\nonumber
\le\ &C\|\Lambda^{1+\beta}a \|_{L^{2}}\|\Lambda^{2-\beta}(a,u)\|_{L^{\frac{1}{\frac{3}{2}-2\beta}}}\|(a,u)\|_{L^{\frac{1}{2\beta-1}}}\\\nonumber
\le\ &C\|\Lambda^{1+\beta}a\|_{L^{2}}\|\Lambda^{1+\beta}(a,u)\|_{L^{2}}\|\Lambda^{\frac{3}{2}-2\beta}(a,u)\|_{L^{2}}\\\nonumber
\le\ &CE_{0}^{\beta-\frac{1}{4}}E_{1}^{\frac{3}{4}-\beta} \mathscr{D}_{1}\\\nonumber
\le\ &\frac{1}{100}\mathscr{D}_{1},
 \end{align}
 and 
 \begin{align}\label{405}
&\frac{1}{2}\frac{d}{dt}\|\Lambda u\|_{L^2}^{2}+\gamma\langle\Lambda a_x,\Lambda u\rangle+\|\Lambda^{1+\beta}u\|_{L^2}^2\\\nonumber
=&\langle \Lambda (k(a)a_x),\Lambda u\rangle-\langle \Lambda(uu_x),\Lambda u\rangle+\left\langle\Lambda\left(\frac{a}{1+a}(-\Delta)^{\beta}u\right),\Lambda u\right\rangle\\\nonumber
\le\ & C\|\Lambda^{1+\beta}u\|_{L^{2}}(\|\Lambda^{1-\beta}a\|_{L^{2}}\|a_x\|_{L^{\infty}}+\|\Lambda^{2-\beta}(a,u)\|_{L^{\frac{1}{\frac{3}{2}-2\beta}}}\| (a,u)\|_{L^{\frac{1}{2\beta-1}}}\\\nonumber
&+\|\Lambda^{2\beta}u\|_{L^{\frac{1}{\beta-\frac{1}{2}}}}\|\Lambda^{1-\beta}a\|_{L^{\frac{1}{1-\beta}}}+\|\Lambda^{1+\beta}u\|_{L^{2}}\|(a,u)\|_{L^{\infty}})\\\nonumber
\le\ &C\|\Lambda^{1+\beta}u\|_{L^{2}}\left(\|(a,u)\|_{L^{2}}^{\frac{6\beta-1}{2(1+\beta)}}\|\Lambda^{1+\beta} (a,u)\|_{L^{2}}^{\frac{5-2\beta}{2(1+\beta)}}+\|\Lambda^{1+\beta}a\|_{L^{2}}\|\Lambda^{\frac{3}{2}-2\beta}a\|_{L^{2}}\right.\\\nonumber
&\left.+\|\Lambda^{1+\beta}u\|_{L^{2}}\|\Lambda^{\frac{1}{2}}a\|_{L^{2}}+\|\Lambda^{1+\beta}u\|_{L^{2}}\|(a,u)\|_{L^{\infty}}\right)\\\nonumber
\le\ &\frac{1}{100}\mathscr{D}_{1}.
\end{align}
 Combining (\ref{404}) and (\ref{405}), we find
\begin{align}\label{406}
\frac{d}{dt}E_{1}+2\|\Lambda^{1+\beta}u\|_{L^2}^2\le \frac{1}{50}\mathscr{D}_{1}.
\end{align}
From similar arguments to (\ref{306}), by Lemma \ref{lemma3}, \ref{lemma5} and taking $\delta$ small enough, we have
\begin{align}\label{407}
&\frac{d}{dt}\langle\Lambda^{\beta}a_x,k\Lambda^{\beta}u\rangle+k\gamma\|\Lambda^{\beta}a_x\|_{L^2}^2\\\nonumber
=& k\langle\Lambda^{\beta}u_x,\Lambda^{\beta}u_x\rangle+k\langle\Lambda^{\beta}(au)_x,\Lambda^{\beta}u_x\rangle-k\langle\Lambda^{\beta}(-\Delta)^{\beta}u,\Lambda^{\beta} a_x\rangle\\\nonumber
&+k\langle \Lambda^{\beta} (K(a)a_x),\Lambda^{\beta}a_x\rangle-k\langle\Lambda^{\beta} (uu_x),\Lambda^{\beta} a_x\rangle+k\left\langle\Lambda^{\beta}\left(\frac{a}{1+a}(-\Delta)^{\beta}u\right),\Lambda^{\beta}a_x\right\rangle\\\nonumber
\le\ &Ck\|\Lambda^{1+\beta}u\|_{L^2}^2+Ck\|(a,u)\|_{L^{\infty}}\|\Lambda^{1+\beta}(a,u)\|_{L^{2}}^{2}+k\|\Lambda^{3\beta}u\|_{L^2}\|\Lambda^{1+\beta}a\|_{L^2}\\\nonumber
&+ Ck\|\Lambda^{\beta}(a,u)\|_{L^{2}}\|\Lambda^{1+\beta}(a,u)\|_{L^{2}}\|\partial_{x}(a,u)\|_{L^{\infty}}+ Ck\|\Lambda^{1+\beta}a\|_{L^{2}}\|a\|_{L^{\infty}}\|\Lambda^{3\beta}u\|_{L^{2}}\\\nonumber
&+Ck\|\Lambda^{1+\beta}a\|_{L^{2}}\|\Lambda^{\beta}a\|_{L^{\frac{1}{1-\beta}}}\|\Lambda^{2\beta}u\|_{L^{\frac{1}{\beta-\frac{1}{2}}}}\\\nonumber
\le\ &Ck\|\Lambda^{1+\beta}u\|_{L^2}^2+k\|(a,u)\|_{L^{2}}^{\frac{1}{2}}\|\partial_{x} (a,u)\|_{L^{2}}^{\frac{1}{2}}\|\Lambda^{1+\beta}(a,u)\|_{L^{2}}^{2}+Ck\|\Lambda^{3\beta}u\|_{L^2}\|\Lambda^{1+\beta}a\|_{L^2}\\\nonumber
&+ Ck\|(a,u)\|_{L^{2}}^{\frac{2\beta+1}{2(1+\beta)}}\|\Lambda^{1+\beta} (a,u)\|_{L^{2}}^{\frac{5+4\beta}{2(1+\beta)}}+ Ck\|\Lambda^{1+\beta}a\|_{L^{2}}\|a\|_{L^{2}}^{\frac{1}{2}}\|a_x\|_{L^{2}}^{\frac{1}{2}}\|\Lambda^{3\beta}u\|_{L^{2}}\\\nonumber
&+Ck\|\Lambda^{1+\beta}(a,u)\|_{L^{2}}^{2}\|a\|_{L^{2}}^{\frac{3}{2}-2\beta}\|a_x\|_{L^{2}}^{2\beta-\frac{1}{2}}\\\nonumber
\le\ &\frac{1}{200}\mathscr{D}_{1}+Ck\|\Lambda^{1+\beta} u\|_{H^{\beta}}^{2}+\frac{k}{100}\|\Lambda^{1+\beta}a\|_{H^{1-\beta}}^{2}.
\end{align}
Taking the $L^2$ inner product of the first two equations of (\ref{eq1}) with $(\sqrt{\gamma}\Lambda^{1+\beta} a,\Lambda^{1+\beta} u)$ and using Lemma \ref{lemma4} and \ref{lemma5}. Similar to the estimates (\ref{404}) and (\ref{405}), one can derive
\begin{align}\label{408}
&\frac{1}{2}\frac{d}{dt}\|\sqrt{\gamma} \Lambda^{1+\beta} a\|_{L^2}^{2}+\gamma\langle\Lambda^{1+\beta}u_x,\Lambda^{1+\beta} a\rangle\\\nonumber
=&-\gamma\langle\Lambda^{1+\beta}(au)_x,\Lambda^{1+\beta} a\rangle\\\nonumber
=&-\gamma\langle[\Lambda^{1+\beta},u\partial_x]a,\Lambda^{1+\beta}a\rangle+\frac{\gamma}{2}\langle u_x, |\Lambda^{1+\beta}a|^{2}\rangle-\gamma\langle\Lambda^{2\beta}(au_x),\Lambda^{2}a\rangle\\\nonumber
\le\ &C\|\Lambda^{1+\beta}(a,u) \|_{L^{2}}^{2}\|\partial_x(a,u)\|_{L^{\infty}}+C\|\Lambda^{2}a \|_{L^{2}}\|\Lambda^{1+2\beta}u \|_{L^{2}}\|a \|_{L^{\infty}}+C\|\Lambda^{2\beta}a\|_{L^{2}}\|\Lambda^{2}a\|_{L^{2}}\|u_x\|_{L^{\infty}}\\\nonumber
\le\ &\frac{1}{100}\mathscr{D}_{1},
 \end{align}
  and
 \begin{align}\label{409}
&\frac{1}{2}\frac{d}{dt}\|\Lambda^{1+\beta} u\|_{L^2}^{2}+\gamma\langle\Lambda^{1+\beta}a_x,\Lambda^{1+\beta} u\rangle+\|\Lambda^{1+2\beta}u\|_{L^2}^2\\\nonumber
=&\langle \Lambda^{1+\beta} (k(a)a_x),\Lambda^{1+\beta} u\rangle-\langle \Lambda^{1+\beta}(u
 u_x),\Lambda^{1+\beta} u\rangle+\left\langle\Lambda^{1+\beta}\left(\frac{a}{1+a}(-\Delta)^{\beta}u\right),\Lambda^{1+\beta} u\right\rangle\\\nonumber
\le\ &C \|\Lambda^{1+2\beta}u\|_{L^{2}}(\|\Lambda(a,u)\|_{L^{2}}\|\partial_x(a,u)\|_{L^{\infty}}+\|\Lambda^{2}(a,u)\|_{L^{2}}\|(a,u) \|_{L^{\infty}}+\|\Lambda^{1+\beta}u\|_{L^{2}}\|\nabla u\|_{L^{\infty}}\\\nonumber
&\|\Lambda^{2\beta}u\|_{L^{\frac{1}{1-\beta}}}\|\Lambda a\|_{L^{\frac{1}{\beta-\frac{1}{2}}}}+\|\Lambda^{1+2\beta}u\|_{L^{2}}\|(a,u)\|_{L^{\infty}})\\\nonumber
\le\ & C\|\Lambda^{1+2\beta}u\|_{L^{2}}(\|\Lambda(a,u)\|_{L^{2}}\|\partial_x(a,u)\|_{L^{\infty}}+\|\Lambda^{2}(a,u)\|_{L^{2}}\|(a,u) \|_{L^{\infty}}+\|\Lambda^{1+\beta}u\|_{L^{2}}\|u_x\|_{L^{\infty}}\\\nonumber
&\|\Lambda^{3\beta-\frac{1}{2}}u\|_{L^2}\|\Lambda^{2-\beta}a\|_{L^{2}}+\|\Lambda^{1+2\beta}u\|_{L^{2}}\|(a,u)\|_{L^{\infty}})\\\nonumber
\le\ & \frac{1}{100}\mathscr{D}_{1}.
\end{align}
Combining (\ref{408}) and (\ref{409}), we find
\begin{align}\label{411}
\frac{d}{dt}E_{1+\beta}+2\|\Lambda^{1+2\beta}u\|_{L^2}^2\le \frac{1}{50}\mathscr{D}_{1}.
\end{align}
Along the same line, we find
\begin{align}\label{412}
&\frac{d}{dt}\langle\Lambda a_x,k\Lambda u\rangle+k\gamma\|\Lambda a_x\|_{L^2}^2\\\nonumber
=& k\langle\Lambda u_x,\Lambda u_x\rangle+k\langle\Lambda(au)_x,\Lambda u_x\rangle-k\langle\Lambda(-\Delta)^{\beta}u,\Lambda a_x\rangle\\\nonumber
&+k\langle \Lambda (K(a)a_x),\Lambda a_x\rangle-k\langle\Lambda(uu_x),\Lambda a_x\rangle+k\left\langle\Lambda\left(\frac{a}{1+a}(-\Delta)^{\beta}u\right),\Lambda a_x\right\rangle\\\nonumber
\le\  &Ck\|\Lambda^{2}u\|_{L^2}^2+Ck\|\Lambda^{1+2\beta}u\|_{L^{2}}\|(a,u)\|_{L^{\frac{1}{2\beta-1}}}\|\Lambda^{3-2\beta}(a,u)\|_{L^{\frac{1}{\frac{3}{2}-2\beta}}}+Ck\|\Lambda^{1+2\beta}u\|_{L^2}\|\Lambda^{2}a\|_{L^2}\\\nonumber
&+Ck\|\Lambda^{2}(a,u)\|_{L^{2}}^{2}\|(a,u)\|_{L^{\infty}}+ Ck\|\Lambda(a,u)\|_{L^{2}}\|\Lambda^{2}(a,u)\|_{L^{2}}\|\partial_{x}(a,u)\|_{L^{\infty}}\\\nonumber
&+ Ck\|\Lambda^{2}a\|_{L^{2}}\|a\|_{L^{\infty}}\|\Lambda^{1+2\beta}u\|_{L^{2}}+Ck\|\Lambda^{2}a\|_{L^{2}}\|\Lambda a\|_{L^{\frac{1}{1-\beta}}}\|\Lambda^{2\beta}u\|_{L^{\frac{1}{\beta-\frac{1}{2}}}}\\\nonumber
\le\ &Ck\|\Lambda^{2}u\|_{L^2}^2+k\|\Lambda^{1+2\beta}u\|_{L^{2}}\|\Lambda^{\frac{3}{2}-2\beta}(a,u)\|_{L^{2}}\|\Lambda^{2}(a,u)\|_{L^{2}}+Ck\|\Lambda^{1+2\beta}u\|_{L^2}\|\Lambda^{2}a\|_{L^2}\\\nonumber
&+Ck\|\Lambda^{2}(a,u)\|_{L^{2}}^{2}\|(a,u)\|_{L^{\infty}}+ Ck\|\Lambda(a,u)\|_{L^{2}}\|\Lambda^{2}(a,u)\|_{L^{2}}\|\partial_x(a,u)\|_{L^{\infty}}\\\nonumber
&+Ck\|\Lambda^{2}a\|_{L^{2}}\|a\|_{L^{\infty}}\|\Lambda^{1+2\beta}u\|_{L^{2}}+Ck\|\Lambda^{2}a\|_{L^{2}}\|\Lambda^{\beta+\frac{1}{2}} a\|_{L^{2}}\|\Lambda^{1+\beta}u\|_{L^{2}}\\\nonumber
\le\ &\frac{1}{200}\mathscr{D}_{1}+Ck\|\Lambda^{1+\beta} u\|_{H^{\beta}}^{2}+\frac{k}{100}\|\Lambda^{1+\beta}a\|_{H^{1-\beta}}^{2}.
\end{align}
By (\ref{406}), (\ref{407}),  (\ref{411}) and (\ref{412}) and taking $k$ small enough, we have
\begin{align}\label{413}
\quad\frac{d}{dt}\mathscr{E}_{1}+\mathscr{D}_{1}\leq 0.
\end{align}
Note
\begin{align*}
\mathscr{E}_{1}\approx \|\nabla(a,u)\|_{H^{\beta}}^{2}.
\end{align*}
Integrating (\ref{413}) in time on $[0,t]$, then we conclude that
\begin{align*}
\sup_{t}\|\nabla(a,u)\|_{H^{\beta}}^{2}+\int_{0}^{t}\mathscr{D}_{1}(t^{\prime})dt^{\prime}\le C\|\nabla(a_0,u_0)\|_{H^{\beta}}^{2}\le \frac{\delta^2(M)}{4}
\end{align*}
holds for small enough $\delta(M)>0$. This together with (\ref{403}) implies the conclusion holds by standard bootstrap argument.
\end{proof}
Now we establish the following decay estimate for $\|\partial_x(a,u)\|_{H^{\beta}}$.
\begin{lemma}\label{4lemm2}
Under the assumptions of Theorem \ref{1theo3}, it holds that
\begin{align}\label{420}
\|\partial_x(a,u)\|_{H^{\beta}}\le C(1+t)^{-\frac{1}{2\beta}}.
\end{align}
\end{lemma}
\begin{proof}
By interpolation between Sobolev spaces, we get
\begin{align*}
\|\Lambda a\|_{L^{2}}^{2}\le C\|a\|_{L^2}^{\frac{2\beta}{1+\beta}}\|\Lambda^{1+\beta}a\|_{L^{2}}^{\frac{2}{1+\beta}}.
\end{align*}
Thanks to Theorem \ref{1theo3}, we obtain
\begin{align*}
\|\Lambda^{1+\beta}a\|_{L^{2}}^{2}\ge C\|\Lambda a\|_{L^{2}}^{2(1+\beta)},
\end{align*}
which together with Theorem \ref{1theo3} gives rise to
\begin{align}\label{421}
\frac{k\gamma}{2}\|\Lambda^{1+\beta}a\|_{H^{1-\beta}}^{2}\ge   C\|\Lambda^{1}a\|_{H^{1}}^{2(1+\beta)}.
\end{align}
Along the same line, we deduce
\begin{align}\label{422}
\|\Lambda^{1+\beta}u\|_{H^{\beta}}^{2}\ge  C\|\Lambda u\|_{H^{\beta}}^{2(1+\beta)}.
\end{align}
By inserting the estimates (\ref{421}) and (\ref{422}) into (\ref{413}), we obtain
\begin{align*}
\frac{d}{dt}\mathscr{E}_{1}+C\mathscr{E}_{1}^{1+\beta}\le 0.
\end{align*}
Then we arrive at
\begin{align*}
\mathscr{E}_{1}\le (1+t)^{-\frac{1}{\beta}}.
\end{align*}
Therefore, the desired bound (\ref{420}) holds true.
\end{proof}
Based on the Lemma \ref{4lemm2}, we are now ready to obtain $H^1$ optimal time decay rate. From this point on, we agree that all occurrences of $\delta$ and $C_2$ denote their positive powers. 
\begin{prop}{\label{4prop3}}
  Under the same conditions as in Theorem \ref{1theo3}, if additionally $\left(a_{0}, u_{0}\right) \in \dot{B}_{2, \infty}^{-\frac{1}{2}}$, then there exists $C>0$ such that for every $t>0$, there holds
\begin{align*}
\|\Lambda^{s_1}(a,u)\|_{H^{1+\beta-s_{1}}}\le C(1+t)^{-\frac{1+2s_{1}}{4\beta}},
\end{align*}
where $0\le s_{1}\le 1$.
\end{prop}
\begin{proof}
It follows readily from (\ref{303})-(\ref{307}) that (\ref{309}) remains valid under the assumptions of Proposition \ref{4prop3}.
  Recall $S(t)=\left\{\xi:|\xi|^{2\beta}\le C_{2}(1+t)^{-1}\right\}$. By (\ref{309}), we get
  \begin{align}\label{423}
  \frac{d}{dt}\mathscr{E}_{0}(t)+\frac{C_{2}}{1+t}\left(k\gamma\|a\|_{H^{\beta}}^{2}+\|u\|_{H^{\beta}}^{2}\right)\leq \frac{C}{1+t}{\int_{S(t)}|\widehat{a}(\xi)|^{2}+|\widehat u}(\xi)|^{2}d\xi.
  \end{align}  
 The estimate of term on the right-hand side of (\ref{423}) follows a similar derivation of (\ref{315}). we have
\begin{align}\label{424}
\gamma|\widehat{a}|^{2}+|\widehat{u}|^{2}\le C(|\widehat{a}_{0}|^{2}+|\widehat{u}_{0}|^{2})+C\int_{0}^{t}|\widehat{F}\cdot\overline{\widehat{a}}|+|\widehat{H}\cdot\overline{\widehat{u}}|dt^{\prime}.
\end{align}
Along the same lines as (\ref{316}), we find
\begin{align}\label{425}
\int_{S(t)}\left(|\widehat{a_{0}}|^{2}+|\widehat{u_{0}}|^{2}\right)d\xi 
 \le C(1+t)^{-\frac{1}{2\beta}}\|(a_{0},u_{0})\|_{\dot{B}_{2,\infty}^{-\frac{1}{2}}}^{2}
\end{align}
Thanks to $\frac{1}{2}<\beta<\frac{3}{4}$, we arrive at
\begin{align}\label{426}
\int_{S(t)}\int_{0}^{t}|\widehat{F}\cdot\overline{\widehat{a}}|+|\widehat{H}\cdot\overline{\widehat{u}}|dt^{\prime}d\xi=& \int_{0}^{t}\int_{S(t)}|\widehat{F}\cdot\overline{\widehat{a}}|+|\widehat{H}\cdot\overline{\widehat{u}}|d\xi dt^{\prime} \\\nonumber
\le\ &C |S(t)|^{\frac{1}{2}}\int_{0}^{t}\|F\|_{L^1}\|a\|_{L^2}+\|H\|_{L^1}\|u\|_{L^2}dt^{\prime}\\\nonumber
\le\ &C\ (1+t)^{-\frac{1}{4\beta}}\int_{0}^{t}\left(\|a\|_{L^2}^2+\|u\|_{L^2}^2\right)\left(\|(a_x,u_x)\|_{L^2}+\|\Lambda^{2\beta}u\|_{L^2}\right)dt^{\prime}\\\nonumber
\le\ &C(1+t)^{-\frac{1}{4\beta}}\int_{0}^{t}(1+t^{\prime})^{-\frac{1}{2\beta}}dt^{\prime}\\\nonumber
\le\ &C(1+t)^{-\frac{3}{4\beta}+1}.\nonumber
\end{align}
Inserting (\ref{425}) and (\ref{426}) into (\ref{424}), we have
\begin{align}\label{427}
{\int_{S(t)}|\widehat{a}(\xi)|^{2}+|\widehat u}(\xi)|^{2}d\xi\le  C(1+t)^{-\frac{3}{4\beta}+1}.
\end{align}
Plugging (\ref{427}) into (\ref{423}) gives rise to 
\begin{align*}
  \frac{d}{dt}\mathscr{E}_{0}(t)+\frac{C_{2}}{1+t}\left(k\gamma\|a\|_{H^{\beta}}^{2}+\|u\|_{H^{\beta}}^{2}\right)\le C(1+t)^{-\frac{3}{4\beta}}.
  \end{align*}  
Consequently, we get the initial time decay rate
\begin{align}\label{428}
\mathscr{E}_{0}(t)\le C(1+t)^{-\frac{3}{4\beta}+1}.
  \end{align}
It follows from (\ref{423})-(\ref{426}) that
\begin{align*}
  \frac{d}{dt}\mathscr{E}_{0}(t)&+\frac{C_{2}}{1+t}\left(k\gamma\|a\|_{H^{\beta}}^{2}+\|u\|_{H^{\beta}}^{2}\right)\le C (1+t)^{-1-\frac{1}{2\beta}}\\\nonumber
  &+C(1+t)^{-1-\frac{1}{4\beta}}\int_{0}^{t}\left(\|a\|_{L^2}^2+\|u\|_{L^2}^2\right)\left(\|(a_x,u_x)\|_{L^2}+\|\Lambda^{2\beta}u\|_{L^2}\right)dt^{\prime},\nonumber
  \end{align*} 
which implies
\begin{align}\label{429}
(1+t)^{1+\frac{1}{4\beta}}\frac{d}{dt}\mathscr{E}_{0}(t)&+C_{2}(1+t)^{\frac{1}{4\beta}}\left(k\gamma\|a\|_{H^{\beta}}^{2}+\|u\|_{H^{\beta}}^{2}\right)\\\nonumber
&\le C(1+t)^{-\frac{1}{4\beta}}+C\int_{0}^{t}\left(\|a\|_{L^2}^2+\|u\|_{L^2}^2\right)\left(\|(a_x,u_x)\|_{L^2}+\|\Lambda^{2\beta}u\|_{L^2}\right)dt^{\prime}.\nonumber
\end{align}
Integrating (\ref{429}) in time on $[0,t]$, we get
\begin{align*}
(1+t)^{\frac{1}{4\beta}}\mathscr{E}_{0}(t)\le C+C\int_{0}^{t}\left(\|a\|_{L^2}^2+\|u\|_{L^2}^2\right)\left(\|(a_x,u_x)\|_{L^2}+\|\Lambda^{2\beta}u\|_{L^2}\right)dt^{\prime}.
\end{align*}
We set $\displaystyle N(t)=\sup_{0\le t^{\prime}\le t}(1+t^{\prime})^{\frac{1}{4\beta}}\mathscr{E}_{0}(t^{\prime})$, then we find
\begin{align*}
N(t)\le C+C\int_{0}^{t}(1+t^{\prime})^{-\frac{1}{4\beta}}N(t^{\prime})\left(\|(a_x,u_x)\|_{L^2}+\|\Lambda^{2\beta}u\|_{L^2}\right)dt^{\prime}.
\end{align*}
Applying Gronwall's inequality yields for any $t>0$, $N(t)<\infty$, which gives rise to
\begin{align}\label{430}
  \mathscr{E}_{0}(t)\le C(1+t)^{-\frac{1}{4\beta}}.
\end{align}
From (\ref{430}), we know that
  \begin{align}\label{431}
  \frac{d}{dt}\mathscr{E}_{1}(t)+\frac{C_{2}}{1+t}\left(k\gamma\|a_x\|_{H^{\beta}}^{2}+\|u_x\|_{H^{\beta}}^{2}\right)&\leq \frac{C}{1+t}\int_{S(t)}|\xi|^{2}\left(|\widehat{a}(\xi)|^{2}+|\widehat {u}(\xi)|^{2}\right)d\xi\\\nonumber
  &\leq C(1+t)^{-\frac{5}{4\beta}-1}.
  \end{align}  
  Then we deduce that
  \begin{align}\label{432}
  \mathscr{E}_{1}\le C(1+t)^{-\frac{5}{4\beta}}.
\end{align}
Next we will prove the solution of (\ref{eq1}) belongs to some negative index Besov space. From (\ref{3105})-(\ref{3111}), we find
\begin{align*}
M^{2}(t)\leq CM^{2}(0)+CM(t)\int_{0}^{t}\left(\|F\|_{\dot{B}^{-\frac{1}{2}}_{2,\infty}}+\|H\|_{\dot{B}^{-\frac{1}{2}}_{2,\infty}}\right)dt^{\prime},
\end{align*}
where $\displaystyle M(t)=\sup_{0\le t^{\prime}\le t}\left(\gamma\|a\|_{\dot{B}^{-\frac{1}{2}}_{2,\infty}}+\|u\|_{\dot{B}^{-\frac{1}{2}}_{2,\infty}}\right)$.
Using the fact $ L^{1}\hookrightarrow\dot{B}^{-\frac{1}{2}}_{2,\infty}$, (\ref{430}) and (\ref{432}), we conclude, for any $t> 0$, that
\begin{align}\label{433}
\int_{0}^{t}\left(\|F\|_{\dot{B}^{-\frac{1}{2}}_{2,\infty}}+\|H\|_{\dot{B}^{-\frac{1}{2}}_{2,\infty}}\right)dt^{\prime}&\le C  \int_{0}^{t}\left(\|F\|_{L^1}+\|H\|_{L^1}\right)dt^{\prime}\\\nonumber
&\le C  \int_{0}^{t}(\|a\|_{L^2}+\|u\|_{L^2})\left(\|(a_x,u_x)\|_{L^2}+\|\Lambda^{2\beta} u\|_{L^2}\right)dt^{\prime}\\\nonumber
&\le C  \int_{0}^{t}(1+t^{\prime})^{-\frac{3}{4\beta}}dt^{\prime}<+\infty.
\end{align}
Hence, we deduce $M(t)<C$. From this we can obtain the optimal time decay rate for $E_{0}$. Using (\ref{433}), we arrive at
\begin{align}\label{434}
\int_{S(t)}\int_{0}^{t}|\widehat{F}\cdot\overline{\widehat{a}}|+|\widehat{H}\cdot\overline{\widehat{u}}|dt^{\prime}d\xi=& \int_{0}^{t}\int_{S(t)}|\widehat{F}\cdot\overline{\widehat{a}}|+|\widehat{H}\cdot\overline{\widehat{u}}|d\xi dt^{\prime} \\\nonumber
\le\ & C\int_{0}^{t}\left(\|F\|_{L^1}\int_{S(t)}|{\widehat{a}}|d\xi+\|H\|_{L^1}\int_{S(t)}|{\widehat{u}}|d\xi\right
) dt^{\prime} \\\nonumber
\le\ &  C(1+t)^{-\frac{1}{4\beta}}\int_{0}^{t}\left[\left(\|F\|_{L^1}+\|H\|_{L^1}\right)\left(\int_{S(t)}|\widehat{a}|^{2}+|\widehat{u}|^{2}d\xi\right)^{\frac{1}{2}}\right]dt^{\prime}\\\nonumber
\le\  &C(1+t)^{-\frac{1}{2\beta}},\nonumber
\end{align}
where we use the fact
\begin{align*}
\left(\int_{S(t)}|\widehat{a}|^{2}+|\widehat{u}|^{2}d\xi\right)^{\frac{1}{2}}dt^{\prime}\le C(1+t)^{-\frac{1}{4\beta}}M(t).
\end{align*}
Thanks to (\ref{423}), (\ref{425}) and (\ref{434}), we derive that
 \begin{align}\label{435}
  \frac{d}{dt}\mathscr{E}_{0}(t)+\frac{C_{2}}{1+t}\left(k\gamma\|a\|_{H^{\beta}}^{2}+\|u\|_{H^{\beta}}^{2}\right)\leq C(1+t)^{-\frac{1}{2\beta}-1},
  \end{align} 
  which implies that
\begin{align}\label{436}
 \mathscr{E}_{0}(t)\le  C(1+t)^{-\frac{1}{2\beta}}.
  \end{align}
  According to (\ref{431}) and (\ref{436}), we infer that
  \begin{align}\label{437}
 \mathscr{E}_{1}(t)\le C(1+t)^{-\frac{3}{2\beta}}.
  \end{align}
 Combining (\ref{436}), (\ref{437}) and Sobolev interpolation, we thus complete the proof of Proposition \ref{4prop3}.
  \end{proof}
\textbf{Proof of Theorem \ref{1theo4}:}
\begin{proof}
  By (\ref{408}) and (\ref{409}), we know
  \begin{align}\label{438}
  &\frac{1}{2}\frac{d}{dt}\|\sqrt{\gamma} \Lambda^{1+\beta}a\|_{L^2}^{2}+\gamma\langle\Lambda^{1+\beta}u_x,\Lambda^{1+\beta} a\rangle\\\nonumber
  =&-\gamma\langle\Lambda^{1+\beta}(au)_x,\Lambda^{1+\beta} a\rangle\\\nonumber
=&-\gamma\langle[\Lambda^{1+\beta},u\partial_x]a,\Lambda^{1+\beta}a\rangle+\frac{\gamma}{2}\langle u_x, |\Lambda^{1+\beta}a|^{2}\rangle-\gamma\langle\Lambda^{2\beta}(au_x),\Lambda^{2}a\rangle\\\nonumber
\le &\ C\|\Lambda^{1+\beta}a\|_{L^{2}}^{2}\|u_x\|_{L^{\infty}}+C\|\Lambda^{1+\beta}a\|_{L^{2}}^{2}\| u_x\|_{L^{\infty}}\\\nonumber
&+C\|\Lambda^{2}a \|_{L^{2}}\|\Lambda^{1+2\beta}u \|_{L^{2}}\|a \|_{L^{\infty}}+C\|\Lambda^{2\beta}a\|_{L^{2}}\|\Lambda^{2}a\|_{L^{2}}\| u_x\|_{L^{\infty}}\\\nonumber
\le &\ C(1+t)^{-\frac{1}{\beta}}\mathscr{D}_{1}+\frac{1}{100}\|\Lambda^{1+2\beta}u \|_{L^{2}}^{2},
  \end{align}
  and
 \begin{align}\label{439}
&\frac{1}{2}\frac{d}{dt}\|\Lambda^{1+\beta} u\|_{L^2}^{2}+\gamma\langle\Lambda^{1+\beta}a_x,\Lambda^{1+\beta} u\rangle+\|\Lambda^{1+2\beta}u\|_{L^2}^2\\\nonumber
=&\langle \Lambda^{1+\beta} (k(a)a_x),\Lambda^{1+\beta} u\rangle-\langle \Lambda^{1+\beta}(u
 u_x),\Lambda^{1+\beta} u\rangle+\left\langle\Lambda^{1+\beta}\left(\frac{a}{1+a}(-\Delta)^{\beta}u\right),\Lambda^{1+\beta} u\right\rangle\\\nonumber
\le\ &C \|\Lambda^{1+2\beta}u\|_{L^{2}}(\|\Lambda(a,u)\|_{L^{4}}^{2}+\|\Lambda^{2}(a,u)\|_{L^{2}}\|(a,u) \|_{L^{\infty}}+\|\Lambda^{1+\beta}u\|_{L^{2}}\|u_x\|_{L^{\infty}}\\\nonumber
&\|\Lambda^{2\beta}u\|_{L^{\frac{1}{1-\beta}}}\|\Lambda^{1+\beta}a\|_{L^{\frac{1}{\beta-\frac{1}{2}}}}+\|\Lambda^{1+2\beta}u\|_{L^{2}}\|(a,u)\|_{L^{\infty}})\\\nonumber
\le\ & C\|\Lambda^{1+2\beta}u\|_{L^{2}}\left(\|\Lambda^{\frac{5}{4}}(a,u)\|_{L^{2}}^{2}+\|\Lambda^{2}(a,u)\|_{L^{2}}\|(a,u) \|_{L^{\infty}}+\|\Lambda^{3\beta-\frac{1}{2}}u\|_{L^2}\|\Lambda^{2}a\|_{L^{2}}\right.\\\nonumber
&\left.+\|\Lambda^{1+2\beta}u\|_{L^{2}}\|(a,u)\|_{L^{\infty}}\right)\\\nonumber
\le\ & \left(\frac{1}{100}+\delta\right)\|\Lambda^{1+2\beta}u\|_{L^{2}}^{2}+C\left(\|(a,u)\|_{L^2}^{\frac{4\beta-1}{1+\beta}}\|\Lambda^{1+\beta}(a,u)\|_{L^2}^{\frac{3-2\beta}{1+\beta}}
+\|(a,u) \|_{L^{\infty}}^{2}\right.\\\nonumber
&\left.+\|\Lambda^{3\beta-\frac{1}{2}}u\|_{L^2}^{2}\right)\mathscr{D}_{1}\\\nonumber
\le\ &\left(\frac{1}{100}+\delta\right)\|\Lambda^{1+2\beta}u\|_{L^{2}}^{2}+C(1+t)^{-1}\mathscr{D}_{1}.
\end{align}
From (\ref{439}) and (\ref{439}), it is clear that
\begin{align}\label{440}
\frac{d}{dt}\| \Lambda^{1+\beta}(\sqrt{\gamma}a,u)\|_{L^2}^{2}+\|\Lambda^{1+2\beta}u\|_{L^2}^2&\le C(1+t)^{-1}\mathscr{D}_{1}.
\end{align}
Arguing similarly as (\ref{322}) gives
\begin{align}\label{441}
\frac{1}{1+t}\int_{0}^{t}(1+t^{\prime})^{\frac{3}{2\beta}+1}\mathscr{D}_{1}dt^{\prime}\le C.
\end{align}
Multiplying (\ref{440}) by $(1+t)^{\frac{3}{2\beta}+2}$, one can arrive at
\begin{align}\label{442}
\frac{d}{dt}\left[(1+t)^{\frac{3}{2\beta}+2}\| \Lambda^{1+\beta}(\sqrt{\gamma}a,u)\|_{L^2}^{2}\right]
\le\ &C(1+t)^{\frac{3}{2\beta}+1}\mathscr{D}_{1}+C(1+t)^{\frac{3}{2\beta}+1}\| \Lambda^{1+\beta}(\sqrt{\gamma}a,u)\|_{L^2}^{2}\\\nonumber
\le\ &C(1+t)^{\frac{3}{2\beta}+1}\mathscr{D}_{1},
\end{align}
Integrating (\ref{442}) over $[0,t]$ and using (\ref{441}),
 we infer that
\begin{align*}
(1+t)^{\frac{3}{2\beta}+2}\| \Lambda^{1+\beta}(\sqrt{\gamma}a,u)\|_{L^2}^{2}
\le C\int_{0}^{t}(1+t)^{\frac{3}{2\beta}+1}\mathscr{D}_{1}dt^{\prime},
\end{align*}
so that we get
\begin{align*}
(1+t)^{\frac{3}{2\beta}+1}\| \Lambda^{1+\beta}(\sqrt{\gamma}a,u)\|_{L^2}^{2}
\le \frac{C}{1+t}\int_{0}^{t}(1+t^{\prime})^{\frac{3}{2\beta}+1}\mathscr{D}_{1}dt^{\prime}\le C.
\end{align*}
As a result, it comes out
\begin{align*}
\| \Lambda^{1+\beta}(\sqrt{\gamma}a,u)\|_{L^2}^{2}
\le C(1+t)^{-\frac{3}{2\beta}-1},
\end{align*}
which together with Proposition \ref{4prop3} completes the proof of Theorem \ref{1theo4}. 
\end{proof}

\vspace{3mm}

\textbf{Acknowledgments}
This work was partially supported by the National Natural Science Foundation of
  China (No.12571261).

\vspace{2mm}

\textbf{Conflict of interest.} The authors do not have any possible conflicts of interest.

\vspace{2mm}

\textbf{Data availability statement.}
 Data sharing is not applicable to this article, as no data sets were generated or analyzed during the current study.

\phantomsection
\addcontentsline{toc}{section}{\refname}
\bibliographystyle{abbrv}
\bibliography{Manuscript1209}

@book {Bahouri2011,
    AUTHOR = {Bahouri, Hajer and Chemin, Jean-Yves and Danchin, Rapha\"{e}l},
     TITLE = {Fourier analysis and nonlinear partial differential equations},
    SERIES = {Grundlehren der Mathematischen Wissenschaften [Fundamental
              Principles of Mathematical Sciences]},
    VOLUME = {343},
 PUBLISHER = {Springer, Heidelberg},
      YEAR = {2011},
     PAGES = {xvi+523},
      ISBN = {978-3-642-16829-1},
   MRCLASS = {35-02 (35L72 35Q30 42-02 42B37 76B03 76D03 76N10)},
  MRNUMBER = {2768550},
MRREVIEWER = {Peter R. Massopust},
       DOI = {10.1007/978-3-642-16830-7},
       URL = {https://doi.org/10.1007/978-3-642-16830-7},
}

@article {Schonbek1985,
    AUTHOR = {Schonbek, Maria Elena},
     TITLE = {{$L^2$} decay for weak solutions of the {N}avier-{S}tokes
              equations},
   JOURNAL = {Arch. Rational Mech. Anal.},
  FJOURNAL = {Archive for Rational Mechanics and Analysis},
    VOLUME = {88},
      YEAR = {1985},
    NUMBER = {3},
     PAGES = {209--222},
      ISSN = {0003-9527},
   MRCLASS = {35Q10 (76D05)},
  MRNUMBER = {775190},
MRREVIEWER = {Yoshikazu Giga},
       DOI = {10.1007/BF00752111},
       URL = {http://dx.doi.org/10.1007/BF00752111},
}

@article {Schonbek1991,
    AUTHOR = {Schonbek, Maria E.},
     TITLE = {Lower bounds of rates of decay for solutions to the
              {N}avier-{S}tokes equations},
   JOURNAL = {J. Amer. Math. Soc.},
  FJOURNAL = {Journal of the American Mathematical Society},
    VOLUME = {4},
      YEAR = {1991},
    NUMBER = {3},
     PAGES = {423--449},
      ISSN = {0894-0347},
   MRCLASS = {35Q30 (35B40 76D05)},
  MRNUMBER = {1103459},
MRREVIEWER = {Michael Wiegner},
       DOI = {10.2307/2939262},
       URL = {http://dx.doi.org/10.2307/2939262},
}

@article {Li2011Large,
    AUTHOR = {Li, Hai Liang and Zhang, Ting},
     TITLE = {Large time behavior of isentropic compressible
              {N}avier-{S}tokes system in {$R^3$}},
   JOURNAL = {Math. Methods Appl. Sci.},
  FJOURNAL = {Mathematical Methods in the Applied Sciences},
    VOLUME = {34},
      YEAR = {2011},
    NUMBER = {6},
     PAGES = {670--682},
      ISSN = {0170-4214},
   MRCLASS = {35Q35 (35B40 76N10)},
  MRNUMBER = {2814719},
MRREVIEWER = {Fa-gui Liu},
       DOI = {10.1002/mma.1391},
       URL = {https://doi.org/10.1002/mma.1391},
}

@article {Xu2019,
    AUTHOR = {Xu, Jiang},
     TITLE = {A low-frequency assumption for optimal time-decay estimates to
              the compressible {N}avier-{S}tokes equations},
   JOURNAL = {Comm. Math. Phys.},
  FJOURNAL = {Communications in Mathematical Physics},
    VOLUME = {371},
      YEAR = {2019},
    NUMBER = {2},
     PAGES = {525--560},
      ISSN = {0010-3616},
   MRCLASS = {35Q35 (35B40)},
  MRNUMBER = {4019913},
MRREVIEWER = {Beno\^{\i}t P. Desjardins},
       DOI = {10.1007/s00220-019-03415-6},
       URL = {https://doi.org/10.1007/s00220-019-03415-6},
}

@article {1959On,
    AUTHOR = {Nirenberg, L.},
     TITLE = {On elliptic partial differential equations},
   JOURNAL = {Ann. Scuola Norm. Sup. Pisa Cl. Sci. (3)},
  FJOURNAL = {Annali della Scuola Normale Superiore di Pisa. Classe di
              Scienze. Serie III},
    VOLUME = {13},
      YEAR = {1959},
     PAGES = {115--162},
      ISSN = {0391-173X},
   MRCLASS = {35.00},
  MRNUMBER = {109940},
MRREVIEWER = {L. Garding},
}

@book {2004Feireisl,
	AUTHOR = {Feireisl, Eduard},
	TITLE = {Dynamics of viscous compressible fluids},
	SERIES = {Oxford Lecture Series in Mathematics and its Applications},
	VOLUME = {26},
	PUBLISHER = {Oxford University Press, Oxford},
	YEAR = {2004},
	PAGES = {xii+212},
	ISBN = {0-19-852838-8},
	MRCLASS = {76N10 (35Q35 76N15)},
	MRNUMBER = {2040667},
	MRREVIEWER = {Piotr Bogus\l aw Mucha},
}

@book {1998Lions,
	AUTHOR = {Lions, Pierre Louis},
	TITLE = {Mathematical topics in fluid mechanics. {V}ol. 2},
	SERIES = {Oxford Lecture Series in Mathematics and its Applications},
	VOLUME = {10},
	NOTE = {Compressible models,
	Oxford Science Publications},
	PUBLISHER = {The Clarendon Press, Oxford University Press, New York},
	YEAR = {1998},
	PAGES = {xiv+348},
	ISBN = {0-19-851488-3},
	MRCLASS = {76-02 (35-02 35Q30 76N10)},
	MRNUMBER = {1637634},
	MRREVIEWER = {Denis Serre},
}

@book {1996Lions,
	AUTHOR = {Lions, Pierre Louis},
	TITLE = {Mathematical topics in fluid mechanics. {V}ol. 1},
	SERIES = {Oxford Lecture Series in Mathematics and its Applications},
	VOLUME = {3},
	NOTE = {Incompressible models,
	Oxford Science Publications},
	PUBLISHER = {The Clarendon Press, Oxford University Press, New York},
	YEAR = {1996},
	PAGES = {xiv+237},
	ISBN = {0-19-851487-5},
	MRCLASS = {76-02 (35Q30 35Q35 76D05)},
	MRNUMBER = {1422251},
	MRREVIEWER = {Denis Serre},
}

@article{kato1,
author = {Kato, Tosio and Ponce, Gustavo},
title = {Commutator estimates and the euler and navier-stokes equations},
journal = {Commun. Pure Appl. Math.},
volume = {41},
number = {7},
pages = {891-907},
year = {1988},
}

@article{1979Matsumura,
  title={The initial value problem for the equations of motion of compressible viscous and heat-conductive fluids},
  author={ Matsumura, A  and  Nishida, T },
  journal={Proc. Jpn. Acad.,Ser. A,Math. Sci.},
  volume={55},
  number={9},
  doi = {https://doi.org/10.3792/pjaa.55.337},
  pages={17408},
  year={1979},
}

@article{1980Matsumura,
  title={The initial value problem for the equations of motion of viscous and heat-conductive gases},
  author={ Matsumura, Akitaka  and  Nishida, Takaaki },
  journal={J. Math. Kyoto Univ.},
doi={https://doi.org/10.1215/kjm/1250522322},
  volume={20},
  pages={67-104},
  year={1980},
}

@article{PONCE1985399,
title = {Global existence of small solutions to a class of nonlinear evolution equations},
journal = {Nonlinear Anal. Theory Methods Appl.},
volume = {9},
number = {5},
pages = {399-418},
year = {1985},
issn = {0362-546X},
doi = {https://doi.org/10.1016/0362-546X(85)90001-X},
url = {https://www.sciencedirect.com/science/article/pii/0362546X8590001X},
author = {Gustavo Ponce},
}

@article{xin1988blowup,
  title={Blowup of smooth solutions to the compressible {N}avier-{S}tokes equation with compact density},
  author={Xin, Zhouping},
  journal={Comm. Pure Appl. Math.},
  volume={51},
  number={3},
doi = {https://doi.org/10.1002/(SICI)1097-0312(199803)51:3<229::AID-CPA1>3.0.CO;2-C},
url = {https://onlinelibrary.wiley.com/doi/abs/10.1002/%28SICI%291097-0312%28199803%2951%3A3%3C229%3A%3AAID-CPA1%3E3.0.CO%3B2-C},
  pages={299-440},
  year={1988},
  publisher={Wiley Online Library},
}

@article {MR1779621,
    AUTHOR = {Danchin, R.},
     TITLE = {Global existence in critical spaces for compressible
              {N}avier-{S}tokes equations},
   JOURNAL = {Invent. Math.},
  FJOURNAL = {Inventiones Mathematicae},
    VOLUME = {141},
      YEAR = {2000},
    NUMBER = {3},
     PAGES = {579--614},
      ISSN = {0020-9910,1432-1297},
   MRCLASS = {76N10 (35Q30)},
  MRNUMBER = {1779621},
MRREVIEWER = {Kevin\ R.\ Zumbrun},
       DOI = {10.1007/s002220000078},
       URL = {https://doi.org/10.1007/s002220000078},
}

@article{Guo01012012,
author = {Yan Guo and Yanjin Wang},
title = {Decay of Dissipative Equations and Negative Sobolev Spaces},
journal = {Communications in Partial Differential Equations},
volume = {37},
number = {12},
pages = {2165--2208},
year = {2012},
publisher = {Taylor \& Francis},
doi = {10.1080/03605302.2012.696296},
URL = {https://doi.org/10.1080/03605302.2012.696296},
eprint = {  https://doi.org/10.1080/03605302.2012.696296}
}

@article {MR2675485,
    AUTHOR = {Chen, Qionglei and Miao, Changxing and Zhang, Zhifei},
     TITLE = {Global well-posedness for compressible {N}avier-{S}tokes
              equations with highly oscillating initial velocity},
   JOURNAL = {Comm. Pure Appl. Math.},
  FJOURNAL = {Communications on Pure and Applied Mathematics},
    VOLUME = {63},
      YEAR = {2010},
    NUMBER = {9},
     PAGES = {1173--1224},
      ISSN = {0010-3640,1097-0312},
   MRCLASS = {35B65 (35Q35 76N10)},
  MRNUMBER = {2675485},
MRREVIEWER = {Paolo\ Maremonti},
       DOI = {10.1002/cpa.20325},
       URL = {https://doi.org/10.1002/cpa.20325},
}

@article{DUAN2007220,
title = {Optimal {${L}^p$}–{${L}^q$} convergence rates for the compressible {N}avier–{S}tokes equations with potential force},
journal = {J. Differential Equations},
volume = {238},
number = {1},
pages = {220-233},
year = {2007},
issn = {0022-0396},
doi = {https://doi.org/10.1016/j.jde.2007.03.008},
url = {https://www.sciencedirect.com/science/article/pii/S0022039607000836},
author = {Renjun Duan and Hongxia Liu and Seiji Ukai and Tong Yang}
}

@article {MR4188989,
    AUTHOR = {Xin, Zhouping and Xu, Jiang},
     TITLE = {Optimal decay for the compressible {N}avier-{S}tokes equations
              without additional smallness assumptions},
   JOURNAL = {J. Differential Equations},
  FJOURNAL = {Journal of Differential Equations},
    VOLUME = {274},
      YEAR = {2021},
     PAGES = {543--575},
      ISSN = {0022-0396,1090-2732},
   MRCLASS = {76N06 (35K65 35L65 35Q30)},
  MRNUMBER = {4188989},
MRREVIEWER = {Piotr\ Biler},
       DOI = {10.1016/j.jde.2020.10.021},
       URL = {https://doi.org/10.1016/j.jde.2020.10.021},
}

@article {MR3609245,
    AUTHOR = {Danchin, Rapha\"el and Xu, Jiang},
     TITLE = {Optimal time-decay estimates for the compressible
              {N}avier-{S}tokes equations in the critical {$L^p$} framework},
   JOURNAL = {Arch. Ration. Mech. Anal.},
  FJOURNAL = {Archive for Rational Mechanics and Analysis},
    VOLUME = {224},
      YEAR = {2017},
    NUMBER = {1},
     PAGES = {53--90},
      ISSN = {0003-9527,1432-0673},
   MRCLASS = {35Q35 (35B40 76N10)},
  MRNUMBER = {3609245},
MRREVIEWER = {Olga\ S.\ Rozanova},
       DOI = {10.1007/s00205-016-1067-y},
       URL = {https://doi.org/10.1007/s00205-016-1067-y},
}

@article {MR2679372,
    AUTHOR = {Charve, Fr\'ed\'eric and Danchin, Rapha\"el},
     TITLE = {A global existence result for the compressible
              {N}avier-{S}tokes equations in the critical {$L^p$} framework},
   JOURNAL = {Arch. Ration. Mech. Anal.},
  FJOURNAL = {Archive for Rational Mechanics and Analysis},
    VOLUME = {198},
      YEAR = {2010},
    NUMBER = {1},
     PAGES = {233--271},
      ISSN = {0003-9527,1432-0673},
   MRCLASS = {35Q30 (76N10)},
  MRNUMBER = {2679372},
MRREVIEWER = {Wengu\ Chen},
       DOI = {10.1007/s00205-010-0306-x},
       URL = {https://doi.org/10.1007/s00205-010-0306-x},
}

@article {MR2847531,
    AUTHOR = {Haspot, Boris},
     TITLE = {Existence of global strong solutions in critical spaces for
              barotropic viscous fluids},
   JOURNAL = {Arch. Ration. Mech. Anal.},
  FJOURNAL = {Archive for Rational Mechanics and Analysis},
    VOLUME = {202},
      YEAR = {2011},
    NUMBER = {2},
     PAGES = {427--460},
      ISSN = {0003-9527,1432-0673},
   MRCLASS = {76N10 (35D35 35L45 35L60 35Q35)},
  MRNUMBER = {2847531},
MRREVIEWER = {Magali\ L\'ecureux-Mercier},
       DOI = {10.1007/s00205-011-0430-2},
       URL = {https://doi.org/10.1007/s00205-011-0430-2},
}

@article{li2022,
      title={Non-uniqueness for the hypo-viscous compressible {N}avier-{S}tokes equations}, 
      author={Yachun Li and Peng Qu and Zirong Zeng and Deng Zhang},
      year={2022},
      JOURNAL={arXiv:2212.05844},
      url={https://arxiv.org/abs/2212.05844}, 
}

@article {MR4643428,
    AUTHOR = {Wang, Shu and Zhang, Shuzhen},
     TITLE = {The initial value problem for the equations of motion of
              fractional compressible viscous fluids},
   JOURNAL = {J. Differential Equations},
  FJOURNAL = {Journal of Differential Equations},
    VOLUME = {377},
      YEAR = {2023},
     PAGES = {369--417},
      ISSN = {0022-0396,1090-2732},
   MRCLASS = {35Q70 (35L65 76N10)},
  MRNUMBER = {4643428},
       DOI = {10.1016/j.jde.2023.09.012},
       URL = {https://doi.org/10.1016/j.jde.2023.09.012},
}

@article {MR4897629,
    AUTHOR = {Wang, Shu and Zhang, Shuzhen},
     TITLE = {The initial value problem of the fractional compressible
              {N}avier-{S}tokes-{P}oisson system},
   JOURNAL = {J. Differential Equations},
  FJOURNAL = {Journal of Differential Equations},
    VOLUME = {438},
      YEAR = {2025},
     PAGES = {113359},
      ISSN = {0022-0396,1090-2732},
   MRCLASS = {35Q35 (35A01 35B40 35R11)},
  MRNUMBER = {4897629},
MRREVIEWER = {Yonghui\ Zhou},
       DOI = {10.1016/j.jde.2025.113359},
       URL = {https://doi.org/10.1016/j.jde.2025.113359},
}

@article {MR4549954,
    AUTHOR = {Chen, Yuhui and Li, Minling and Yao, Qinghe and Yao, Zhengan},
     TITLE = {The sharp time-decay rates for one-dimensional compressible
              isentropic {N}avier-{S}tokes and magnetohydrodynamic flows},
   JOURNAL = {Sci. China Math.},
  FJOURNAL = {Science China. Mathematics},
    VOLUME = {66},
      YEAR = {2023},
    NUMBER = {3},
     PAGES = {475--502},
      ISSN = {1674-7283,1869-1862},
   MRCLASS = {35Q30 (35B40)},
  MRNUMBER = {4549954},
MRREVIEWER = {Pavel\ I.\ Naumkin},
       DOI = {10.1007/s11425-021-1937-9},
       URL = {https://doi.org/10.1007/s11425-021-1937-9},
}

@article {MR4884564,
    AUTHOR = {Stefanov, Atanas and Wu, Jiahong and Xu, Xiaojing and Ye,
              Zhuan},
     TITLE = {Global regularity results of the 2{D} fractional {B}oussinesq
              equations},
   JOURNAL = {Math. Ann.},
  FJOURNAL = {Mathematische Annalen},
    VOLUME = {391},
      YEAR = {2025},
    NUMBER = {4},
     PAGES = {5965--6012},
      ISSN = {0025-5831,1432-1807},
   MRCLASS = {35Q35 (35B65 76D03)},
  MRNUMBER = {4884564},
MRREVIEWER = {Qiwei\ Wu},
       DOI = {10.1007/s00208-024-03073-7},
       URL = {https://doi.org/10.1007/s00208-024-03073-7},
}

@article{18M1167905,
author = {Li, Jinkai},
title = {Global Well-Posedness of the One-Dimensional Compressible {N}avier--{S}tokes Equations with Constant Heat Conductivity and Nonnegative Density},
   JOURNAL = {SIAM J. Math. Anal.},
  FJOURNAL = {SIAM Journal on Mathematical Analysis},
volume = {51},
number = {5},
pages = {3666-3693},
year = {2019},
doi = {10.1137/18M1167905},
URL = { https://doi.org/10.1137/18M1167905},
eprint = {   https://doi.org/10.1137/18M1167905}
,
}

@article{Li2016Some,
  title={Some uniform estimates and large-time behavior of solutions to one-dimensional compressible Navier-Stokes system in unbounded domains with large data},
  author={Li, Jing and Liang, Zhilei},
  journal={Arch. Ration. Mech. Anal.},
  volume={220},
  number={3},
  pages={1195--1208},
  year={2016},
  publisher={Springer},
doi = {110.1007/s00205-015-0952-0},
URL = {https://doi.org/10.1007/s00205-015-0952-0}
}

@article{DavidHoff,
 ISSN = {00029947},
 URL = {http://www.jstor.org/stable/2000785},
 author = {David Hoff},
 journal = {Trans. Amer. Math. Soc.},
 number = {1},
 pages = {169--181},
 publisher = {American Mathematical Society},
 title = {Global Existence for $1{D}$, Compressible, Isentropic Navier-Stokes Equations with Large Initial Data},
 urldate = {2026-02-24},
 volume = {303},
 year = {1987}
}

@article{Kanel1979,
  author  = {Kanel', Ya. I.},
  year    = {1979},
  title   = {Cauchy problem for the equations of gasdynamics with viscosity},
  journal = {Sib. Math. J.},
  volume  = {20},
  number  = {2},
  pages   = {208--218},
  doi     = {10.1007/BF00970025}
}

@article{Kazhikhov1979,
  author  = {Kazhikhov, A. V.},
  title   = {Stabilization of solutions of the initial-boundary value problem for barotropic viscous fluid equations},
  journal = {Differ. Uravn.},
  year    = {1979},
  volume  = {15},
  number  = {4},
  pages   = {662--667},
}

@article{KAZHIKHOV1977273,
title = {Unique global solution with respect to time of initial-boundary value problems for one-dimensional equations of a viscous gas},
journal = {J. Appl. Math. Mech.},
volume = {41},
number = {2},
pages = {273-282},
year = {1977},
issn = {0021-8928},
doi = {https://doi.org/10.1016/0021-8928(77)90011-9},
url = {https://www.sciencedirect.com/science/article/pii/0021892877900119},
author = {A.V. Kazhikhov and V.V. Shelukhin}
}

@article {MR1671920,
    AUTHOR = {Jiang, Song},
     TITLE = {Large-time behavior of solutions to the equations of a
              one-dimensional viscous polytropic ideal gas in unbounded
              domains},
   JOURNAL = {Comm. Math. Phys.},
  FJOURNAL = {Communications in Mathematical Physics},
    VOLUME = {200},
      YEAR = {1999},
    NUMBER = {1},
     PAGES = {181--193},
      ISSN = {0010-3616,1432-0916},
   MRCLASS = {35Q35 (35B40 35L60 76N10)},
  MRNUMBER = {1671920},
MRREVIEWER = {Ivan\ Stra\v skraba},
       DOI = {10.1007/s002200050526},
       URL = {https://doi.org/10.1007/s002200050526},
}

@article {MR971685,
    AUTHOR = {Stra\v{s}kraba, Ivan and Valli, Alberto},
     TITLE = {Asymptotic behaviour of the density for one-dimensional
              {N}avier-{S}tokes equations},
   JOURNAL = {Manuscripta Math.},
  FJOURNAL = {Manuscripta Mathematica},
    VOLUME = {62},
      YEAR = {1988},
    NUMBER = {4},
     PAGES = {401--416},
      ISSN = {0025-2611,1432-1785},
   MRCLASS = {35B40 (35Q10 76N10)},
  MRNUMBER = {971685},
       DOI = {10.1007/BF01357718},
       URL = {https://doi.org/10.1007/BF01357718},
}

@article {MR1370096,
    AUTHOR = {Matsumura, Akitaka and Yanagi, Shigenori},
     TITLE = {Uniform boundedness of the solutions for a one-dimensional
              isentropic model system of compressible viscous gas},
   JOURNAL = {Comm. Math. Phys.},
  FJOURNAL = {Communications in Mathematical Physics},
    VOLUME = {175},
      YEAR = {1996},
    NUMBER = {2},
     PAGES = {259--274},
      ISSN = {0010-3616,1432-0916},
   MRCLASS = {35Q35 (35L60 76N15)},
  MRNUMBER = {1370096},
       URL = {http://projecteuclid.org/euclid.cmp/1104275924},
}

@article {MR1890882,
    AUTHOR = {Stra\v{s}kraba, Ivan and Zlotnik, Alexander},
     TITLE = {On a decay rate for 1{D}-viscous compressible barotropic fluid
              equations},
   JOURNAL = {J. Evol. Equ.},
  FJOURNAL = {Journal of Evolution Equations},
    VOLUME = {2},
      YEAR = {2002},
    NUMBER = {1},
     PAGES = {69--96},
      ISSN = {1424-3199,1424-3202},
   MRCLASS = {76N10 (35Q35)},
  MRNUMBER = {1890882},
MRREVIEWER = {Kevin\ R.\ Zumbrun},
       DOI = {10.1007/s00028-002-8080-3},
       URL = {https://doi.org/10.1007/s00028-002-8080-3},
}

@article {MR2812966,
    AUTHOR = {Zhu, Changjiang and Zi, Ruizhao},
     TITLE = {Asymptotic behavior of solutions to 1{D} compressible
              {N}avier-{S}tokes equations with gravity and vacuum},
   JOURNAL = {Discrete Contin. Dyn. Syst.},
  FJOURNAL = {Discrete and Continuous Dynamical Systems},
    VOLUME = {30},
      YEAR = {2011},
    NUMBER = {4},
     PAGES = {1263--1283},
      ISSN = {1078-0947,1553-5231},
   MRCLASS = {35Q35 (35B40 35B45 76N99)},
  MRNUMBER = {2812966},
MRREVIEWER = {S\'ebastien\ J.\ Boyaval},
       DOI = {10.3934/dcds.2011.30.1263},
       URL = {https://doi.org/10.3934/dcds.2011.30.1263},
}

@article {MR4659290,
    AUTHOR = {Chen, Ke and Ha, Ly Kim and Hu, Ruilin and Nguyen, Quoc-Hung},
     TITLE = {Global well-posedness of the 1d compressible {N}avier-{S}tokes
              system with rough data},
   JOURNAL = {J. Math. Pures Appl. (9)},
  FJOURNAL = {Journal de Math\'ematiques Pures et Appliqu\'ees. Neuvi\`eme
              S\'erie},
    VOLUME = {179},
      YEAR = {2023},
     PAGES = {425--453},
      ISSN = {0021-7824,1776-3371},
   MRCLASS = {76N06 (35A01 35D30)},
  MRNUMBER = {4659290},
MRREVIEWER = {Ond\v rej\ Kreml},
       DOI = {10.1016/j.matpur.2023.09.012},
       URL = {https://doi.org/10.1016/j.matpur.2023.09.012},
}

@article{liang2026,
      title={Global regularity and sharp decay to the 2{D} Hypo-Viscous compressible {N}avier-{S}tokes equations}, 
      author={Chen Liang and Zhaonan Luo and Zhaoyang Yin},
      year={2026},
      JOURNAL ={arXiv:2601.06889},
}

@book {MR3185174,
    AUTHOR = {Sato, Keniti},
     TITLE = {L\'evy processes and infinitely divisible distributions},
    SERIES = {Cambridge Studies in Advanced Mathematics},
    VOLUME = {68},
   EDITION = {Revised},
      NOTE = {Translated from the 1990 Japanese original},
 PUBLISHER = {Cambridge University Press, Cambridge},
      YEAR = {2013},
     PAGES = {xiv+521},
      ISBN = {978-1-107-65649-9},
   MRCLASS = {60G51 (60E07 60G18 60G52 60J45)},
  MRNUMBER = {3185174},
}

@article {MR1031939,
    AUTHOR = {Rosenau, Philip},
     TITLE = {Extending hydrodynamics via the regularization of the
              {C}hapman-{E}nskog expansion},
   JOURNAL = {Phys. Rev. A (3)},
  FJOURNAL = {Physical Review. A. Third Series},
    VOLUME = {40},
      YEAR = {1989},
    NUMBER = {12},
     PAGES = {7193--7196},
      ISSN = {1050-2947,1094-1622},
   MRCLASS = {82C40 (76P05)},
  MRNUMBER = {1031939},
MRREVIEWER = {Carlo\ Cercignani},
       DOI = {10.1103/PhysRevA.40.7193},
       URL = {https://doi.org/10.1103/PhysRevA.40.7193},
}

@article {MR2737788,
    AUTHOR = {de Pablo, Arturo and Quir\'os, Fernando and Rodr\'iguez, Ana
              and V\'azquez, Juan Luis},
     TITLE = {A fractional porous medium equation},
   JOURNAL = {Adv. Math.},
  FJOURNAL = {Advances in Mathematics},
    VOLUME = {226},
      YEAR = {2011},
    NUMBER = {2},
     PAGES = {1378--1409},
      ISSN = {0001-8708,1090-2082},
   MRCLASS = {35R11 (35A01 35A02 35B65 35K15 35K57 76S05)},
  MRNUMBER = {2737788},
MRREVIEWER = {Emil\ Popescu},
       DOI = {10.1016/j.aim.2010.07.017},
       URL = {https://doi.org/10.1016/j.aim.2010.07.017},
}

@article {MR2954615,
    AUTHOR = {de Pablo, Arturo and Quir\'os, Fernando and Rodr\'iguez, Ana
              and V\'azquez, Juan Luis},
     TITLE = {A general fractional porous medium equation},
   JOURNAL = {Comm. Pure Appl. Math.},
  FJOURNAL = {Communications on Pure and Applied Mathematics},
    VOLUME = {65},
      YEAR = {2012},
    NUMBER = {9},
     PAGES = {1242--1284},
      ISSN = {0010-3640,1097-0312},
   MRCLASS = {35R11 (35A01 35A02 35B65 35D35 35K57)},
  MRNUMBER = {2954615},
MRREVIEWER = {Erwin\ Topp},
       DOI = {10.1002/cpa.21408},
       URL = {https://doi.org/10.1002/cpa.21408},
}

@article {MR283426,
    AUTHOR = {Itaya, Nobutoshi},
     TITLE = {On the {C}auchy problem for the system of fundamental
              equations describing the movement of compressible viscous
              fluid},
   JOURNAL = {Kodai Math. Sem. Rep.},
  FJOURNAL = {Kodai Mathematical Seminar Reports},
    VOLUME = {23},
      YEAR = {1971},
     PAGES = {60--120},
      ISSN = {0023-2599},
   MRCLASS = {35.79 (76.00)},
  MRNUMBER = {283426},
MRREVIEWER = {Tosio\ Kato},
       URL = {http://projecteuclid.org/euclid.kmj/1138846265},
}

@article {MR421305,
    AUTHOR = {Itaya, Nobutoshi},
     TITLE = {On the initial value problem of the motion of compressible
              viscous fluid, especially on the problem of uniqueness},
   JOURNAL = {J. Math. Kyoto Univ.},
  FJOURNAL = {Journal of Mathematics of Kyoto University},
    VOLUME = {16},
      YEAR = {1976},
    NUMBER = {2},
     PAGES = {413--427},
      ISSN = {0023-608X},
   MRCLASS = {76.35},
  MRNUMBER = {421305},
MRREVIEWER = {W.\ Langlois},
       DOI = {10.1215/kjm/1250522922},
       URL = {https://doi.org/10.1215/kjm/1250522922},
}

@article {MR4491875,
    AUTHOR = {Li, Jinkai and Xin, Zhouping},
     TITLE = {Entropy-bounded solutions to the one-dimensional heat
              conductive compressible {N}avier-{S}tokes equations with far
              field vacuum},
   JOURNAL = {Comm. Pure Appl. Math.},
  FJOURNAL = {Communications on Pure and Applied Mathematics},
    VOLUME = {75},
      YEAR = {2022},
    NUMBER = {11},
     PAGES = {2393--2445},
      ISSN = {0010-3640,1097-0312},
   MRCLASS = {35Q30},
  MRNUMBER = {4491875},
MRREVIEWER = {Jean\ C.\ Cortissoz},
       DOI = {10.1002/cpa.22015},
       URL = {https://doi.org/10.1002/cpa.22015},
}

@article {MR4039142,
    AUTHOR = {Li, Jinkai and Xin, Zhouping},
     TITLE = {Entropy bounded solutions to the one-dimensional compressible
              {N}avier-{S}tokes equations with zero heat conduction and far
              field vacuum},
   JOURNAL = {Adv. Math.},
  FJOURNAL = {Advances in Mathematics},
    VOLUME = {361},
      YEAR = {2020},
     PAGES = {106923, 50},
      ISSN = {0001-8708,1090-2082},
   MRCLASS = {35Q35 (76N10)},
  MRNUMBER = {4039142},
MRREVIEWER = {Luc\ Paquet},
       DOI = {10.1016/j.aim.2019.106923},
       URL = {https://doi.org/10.1016/j.aim.2019.106923},
}

@article {MR2368905,
    AUTHOR = {Mellet, A. and Vasseur, A.},
     TITLE = {Existence and uniqueness of global strong solutions for
              one-dimensional compressible {N}avier-{S}tokes equations},
   JOURNAL = {SIAM J. Math. Anal.},
  FJOURNAL = {SIAM Journal on Mathematical Analysis},
    VOLUME = {39},
      YEAR = {2007/08},
    NUMBER = {4},
     PAGES = {1344--1365},
      ISSN = {0036-1410,1095-7154},
   MRCLASS = {76N10 (35Q30)},
  MRNUMBER = {2368905},
MRREVIEWER = {Vladimir\ V.\ Shelukhin},
       DOI = {10.1137/060658199},
       URL = {https://doi.org/10.1137/060658199},
}



\vspace{5ex}

(Chen Liang)\par\nopagebreak
\noindent\textsc{School of Science,
	Shenzhen Campus of Sun Yat-sen University, Shenzhen 518107, China}

Email address: {\texttt{liangch89@mail2.sysu.edu.cn}}

\vspace{3ex}

(Zhaonan Luo)\par\nopagebreak
\noindent\textsc{School of Science,
	Shenzhen Campus of Sun Yat-sen University, Shenzhen 518107, China}

Email address: {\texttt{luozhn7@mail.sysu.edu.cn}}

\vspace{3ex}

(Zhaoyang Yin)\par\nopagebreak
\noindent\textsc{School of Science,
	Shenzhen Campus of Sun Yat-sen University, Shenzhen 518107, China}

Email address: {\texttt{mcsyzy@mail.sysu.edu.cn}}

\end{document}